\documentclass{amsart}
\usepackage{amssymb,latexsym,graphics,color}
\usepackage[mathscr]{eucal}
\usepackage{fullpage}
\newtheorem{thm}{Theorem}[section]

\newtheorem{lem}[thm]{Lemma}
\newtheorem{prop}[thm]{Proposition}

\theoremstyle{remark}
\newtheorem{rem}[thm]{Remark}

\newtheorem{defn}[thm]{Definition}
\theoremstyle{definition}

\newcommand{\lp}[2]{\Vert \, #1 \, \Vert_{#2}}
\newcommand{\slp}[2]{\Vert   #1  \Vert_{#2}}
\newcommand{\llp}[2]{ |\!|\!| #1 |\!|\!|_{#2}}
\newcommand{\sech}{\hbox{sech}}
\newcommand{\td}[1]{\widetilde{#1}}

\newcommand{\case}[2]{\noindent\textbf{Case #1:}(\emph{#2})}
\newcommand{\step}[2]{\noindent\textbf{Step #1:}(\emph{#2})}

\begin{document}

\title[Variable Cubic Klein-Gordon Scattering]
{Dispersive Decay for the 1D Klein-Gordon Equation with Variable
Coefficient Nonlinearities}
\author{Jacob\ Sterbenz}
% Address of record for the research reported here
\address{Department of Mathematics, University of California San
  Diego, La Jolla, CA 92093-0112} \email{jsterben@math.ucsd.edu}
\address{}
\thanks{The author was supported in part by NSF grant DMS-1001675.}
\date{}
\dedicatory{}
\commby{}

%%% ----------------------------------------------------------------------

\begin{abstract}
We study the 1D  Klein-Gordon equation with variable coefficient 
nonlinearity. This problem exhibits an interesting resonant
interaction between the spatial frequencies of the nonlinear
coefficients and the temporal oscillations of the solutions. In the
case when only the cubic coefficients are variable we prove
$L^\infty$ scattering and   smoothness of the solution in 
weighted spaces with the help of  both quadratic and cubic
normal forms transformations. In the case of cubic interactions
these  normal forms appear to be novel.
\end{abstract}

%%% ----------------------------------------------------------------------
\maketitle
\tableofcontents
%%% ----------------------------------------------------------------------
%%%%%%%%%%%%%%%%%%%%%%%%%%%%%%%%%-----------------------------------------------------------------------

%-------------------------------------------------------------------------
%%%%%%%%%%%%%%%%%%%%%%%%%%%%%% 
%-------------------------------------------------------------------------

\section{Introduction}

In this paper we initiate the study of 1D Klein-Gordon equations with variable
coefficient nonlinearities. Specifically we investigate small compactly supported
solutions to the problem:
\begin{equation}
		(\Box+1)\phi \ = \ \alpha_0 \phi^2 + \beta(x) \phi^3 \ , \qquad \Box \ = \ \partial_t^2 -\partial_x^2
		\ , \label{KG}
\end{equation}
where $\beta(x)=\beta_0 + \beta_1(x)$, with $\beta_1$ is a real valued
Schwartz function, and $\alpha_0,\beta_0 \in \mathbb{R}$.
To motivate this  recall the two well known one dimensional nonlinear
Klein-Gordon equations:
\begin{align}
		\Box \phi  -2\phi+2\phi^3\ &= \ 0 \ , &\hbox{($\phi^4$ model)} \notag\\
		\Box  u  +\sin(u) \ &= \ 0 \ .
		&\hbox{(Sin-Gordon equation)} \notag
\end{align} 
Both of these equations have static ``kink'' type solutions which are respectively (see \cite{Man-S}):
\begin{equation}
		\phi_0 \ = \ \tanh(x) \ , \qquad
		u_0 \ = \ 4 \arctan(e^x) \ . \notag
\end{equation}
Linearization of their respective equations around these solutions leads to 
nonlinear Klein-Gordon equations of the form \eqref{KG}, albeit with potentials
and nonlinear quadratic coefficients as well. The asymptotic stability problem
of small solutions to such generalizations of \eqref{KG} appears to be quite
difficult,
so in the present work we focus attention on the more modest goal of
 understanding the simplified model equation \eqref{KG}. Our main result
 here is the following:

\begin{thm}\label{main_thm}
Let $\phi(t,x)$ be the global solution to the equation \eqref{KG} with sufficiently 
small and smooth compactly supported 
initial data $\phi(0)=\phi_0$ and $\phi_t(0)=\phi_1$.
Then for $t\geqslant 0$ the solution $\phi(t,x)$ obeys the   $L^\infty$ estimate:
\begin{equation}
    |\phi(t,x)| \ \lesssim \ \frac{C(\phi_0,{\phi}_1)}{(1+ |\rho|)^\frac{1}{2}}
    ,\qquad \rho=|t^2-x^2|^{1/2}
    \ . \label{main_decay_est}
\end{equation}
\end{thm}

Notice that solutions the scalar equation \eqref{KG} enjoy the conserved
energy:
\begin{equation}
		E(t) \ = \
		\frac{1}{2}\int_\mathbb{R}
		(\phi_t^2+\phi_x^2+\phi^2 - \frac{2}{3}\alpha_0\phi^3 - \frac{1}{2}\beta \phi^4)(t)dx \ , \notag
\end{equation}
so   global existence from small energy class initial data is not an issue.
On the other hand the dispersion rate for solutions to the 1D linear homogeneous Klein-Gordon
equation $(\Box+1)\phi=0$ is only $t^{-\frac{1}{2}}$, placing the problem \eqref{KG}
well out of reach from the point of view of $L^1\! -\! L^\infty$ type dispersive (or Strichartz) estimates.
Moreover, even using the strongest decay estimates one can muster based on a combination
of the vector-field
method of Klainerman \cite{K1} and  normal forms method of Shatah \cite{Sh}
the problem \eqref{KG}
is still out of reach unless one takes into account the long range behavior  
of the cubic nonlinearity.   In fact a deep result of Delort
\cite{D1} shows that  the global  solutions to \eqref{KG} with constant coefficients
have the asymptotic behavior:
\begin{equation}
 	\phi(t,x)\sim \rho^{-1/2} e^{i\psi(\rho,\,x/\rho)}\, a(x/\rho)+
 	\rho^{-1/2} e^{-i\psi(\rho, x/\rho)}\, \overline{a(x/\rho)} \ , \label{delorts}
 \end{equation}
 were  the phase is not linear but instead given by the expression:
 \begin{equation}
 		\psi(\rho,x/\rho) \ = \ 
		\rho-\Big(\frac{5}{3}\alpha_0^2 +\frac{3}{2}\beta_0 \Big)\sqrt{1-(x/\rho)^2}
		 |a(x/\rho)|^2\,\ln{\rho} \ . \notag
 \end{equation}
In particular this shows that any direct attempt to control solutions of \eqref{KG} globally in time
via Duhamel's principle for the linear Klein-Gordon equation must fail.

In the present paper we push beyond the analysis of \cite{D1} to include the case of 
 variable coefficient cubic nonlinearities of the  form explained above. The main difficulty in
this analysis stems from the fact that the vector-field technique, needed to handle 
the (non-localized) constant coefficient nonlinearities involving $\alpha_0$ and $\beta_0$, is to
some extent incompatible with  coefficients depending on the spatial variable. 
To see this
note that  the linear Klein-Gordon equation $\Box + 1$
is not scale invariant and hence there is no vector-field
that preserves free solutions
based on scalings. The only useful weighted derivative that commutes with the linear
flow appears to be the Lorentz boost $\partial_y=t\partial_x + x\partial_t$, and differentiation by it
of coefficients depending on $x$ leads to very badly diverging sources terms.

At first it might seem as if the lack of good vector-fields   
places   \eqref{KG} well out of reach of the classical methods of Klainerman and Shatah.
However, if one pushes the method of Shatah in a new direction to include not only quadratic
but variable
coefficient cubic normal forms, then  \eqref{KG} can be rewritten
in such a way that it is amenable to estimates involving the Lorentz boost. However, in order
to make this analysis work we must  develop a Littlewood-Paley type calculus
as well as bilinear $\Psi$DO operators and cubic paraproducts
in terms of $\partial_y$, instead of the usual translation derivative $\partial_x$.
This circle of ideas is interesting in its own right and we believe it has applications
 problems which are more general than what we consider here.

The method of the present paper owes much to previous work of Lindblad-Soffer
\cite{L-S1} and \cite{L-S2}
which gave a simplified proof of Delort's global existence theorem from \cite{D1}
in the case when $\alpha_0=\beta_1\equiv 0$. This simplified method
bypasses the phase corrections \eqref{delorts} and provides sharp $L^\infty$ estimates
directly in terms of an asymptotic form of the equation \eqref{KG}. In the context of
our work this method is further clarified by writing the asymptotic equation in a paradifferential
from with respect to $\partial_y$ based frequency cutoffs, in which case the $L^\infty$ estimates
(and even the phase corrections if one wishes) become particularly transparent.\\

The remainder of our paper is organized as follows: In the next section we set 
up hyperboloid coordinates so that $\partial_y$ becomes a coordinate derivative. 
Essentially all of the analysis in the remainder of the paper is carried out in these coordinates.
At this point it is also natural to give a  heuristic overview of the proof before getting into
technical details.

In Section \ref{Norms_sect} we introduce the main function spaces we'll use in our argument, and 
a few of their basic properties. 

In Section \ref{Main_sect} we state the quadratic and cubic normal forms 
transformations we use in an abstract form, as well as some nonlinear energy and $L^\infty$ estimates,
and then use these to derive Theorem \ref{main_thm}. 

In Section \ref{Harm_sect} we set up
all the harmonic analysis used later in the normal forms constructions. This involves a number
of lemmas whose proofs are for the most part independent of the nonlinear problem.

In Section \ref{Quad_sect} we construct and estimate our quadratic 
normal forms transformations, which are defined according 
to the classical method of Shatah \cite{Sh}. 

In Section \ref{Cubic_sect} we 
construct and estimate our variable coefficient
cubic normal forms transformation. 

%-------------------------------------------------------------------------

\subsection{Basic Notation}
We'll use the standard notations $A\lesssim B$, $A\approx B$, $A\gtrsim B$ to denote
$A\leqslant CB$, $C^{-1}B\leqslant B\leqslant CB$, and $CA\geqslant B$ from some
implicit $C>0$. The notation $\langle q\rangle=(1+q^2)^\frac{1}{2} $ will be used
for scalars and operators. We also set $(a)_+=\max\{a,0\}$.

If $N$ is a Banach space with norm $\lp{\cdot}{N}$
and $L$ is an invertible linear operator on $N$ we denote by $L N$  the Banach space
with norm $\lp{u}{LN}=\lp{L^{-1}u}{N}$. More generally if $L:N_1\to N_2$ is continuous then 
we'll often denote this by the inclusion $L N_1\subseteq N_2$, which is consistent
with the previous convention if $L$ is invertible. 

In this work all analysis is done with 
respect to one spatial variable, thus unless otherwise stated $L^p$, $H^1$, etc denote
$L^p(\mathbb{R})$, $H^1(\mathbb{R})$, etc. We denote by $\dot H^1$ the space
with norm $\lp{u}{\dot H^1}=\lp{\partial_x u}{L^2}$, and by $\mathcal{S}=\mathcal{S}(\mathbb{R})$
the one dimensional Schwartz space. We use the convention that for $\varphi\in \mathcal{S}$
a quantity $C(\lp{\varphi}{\mathcal{S}})$ denotes a positive constant depending on finitely many 
Schwartz seminorms of $\varphi$.

%-------------------------------------------------------------------------

\subsection{Acknowledgements}
This work began as a joint project with H. Lindblad,
A. Soffer, and I. Rodnianski. Although all participants made
important contributions to the project, the author assumes full
responsibility for the correctness of the details in this paper.
The author would also like to thank the anonymous referee for
helpful comments.

For further exposition and an alternate presentation  of the technical details,
based on a previous draft of this paper, please refer to \cite{L-S3}.

%-------------------------------------------------------------------------
%%%%%%%%%%%%%%%%%%%%%%%%%%%%%% 
%-------------------------------------------------------------------------

\section{Equations and Coordinates}

Inside the forward light cone $t>|x|$ we set new coordinates:
\begin{equation}
    x \ = \ \rho\sinh(y) \ , \qquad
    t \ =  \ \rho\cosh(y) \ ,
    \qquad\qquad
    \partial_y \ = \ t\partial_x + x\partial_t \ ,
    	\qquad  \partial_\rho \ = \ \rho^{-1}
    	(t\partial_t + x\partial_x) \ . 
     \label{coords}
\end{equation}
In the coordinates \eqref{coords} the Minkowski metric takes the  simple
form $-dt^2 + dx^2 =  - d\rho^2 + \rho^2  dy^2$ 
and therefore, the linear Klein-Gordon operator may be written as:
\begin{equation}
    \Box+1 \ = \ \partial_t^2 - \partial_x^2 + 1  \ = \
    \partial_\rho^2 + \frac{1}{\rho}\partial_\rho - \frac{1}{\rho^2}\partial_y^2
    + 1 \ . \label{eqs}
\end{equation}
Conjugating the RHS above by $\rho^\frac{1}{2}$ we have:
\begin{equation}
		\rho^\frac{1}{2}(\Box+1) \ = \ (\Box_\mathcal{H}+1)\rho^\frac{1}{2} \ ,
		\qquad \Box_\mathcal{H} \ = \ \partial_\rho^2 - \rho^{-2}\partial_y^2 +  
		\frac{1}{4}\rho^{-2} \ . \label{box_conj}
\end{equation}
In particular we may write the equation \eqref{KG} as follows:
\begin{equation}
		(\Box_\mathcal{H}+1)u \ = \ \rho^{-\frac{1}{2}}\alpha_0 u^2 
		+ \rho^{-1}\beta\big(\rho\sinh(y)\big) u^3 \ , \qquad
		\hbox{where\ \ } u \ = \ \rho^\frac{1}{2}\phi \ . 
		\label{KG_hyp}
\end{equation}
Therefore, after translating the initial data  problem forward in time
by a bounded amount, to prove Theorem \ref{main_thm} it suffices to show:

\begin{thm}\label{main_thm_hyp}
Let $u(\rho,y)$ be a global solution to the equation \eqref{KG_hyp} with sufficiently 
small and smooth compactly supported 
initial data $u(1)=u_0$ and $u_\rho(1)=u_1$.
Then for $\rho\geqslant 1$ one has the uniform  $L^\infty$ estimate
$|u|  \lesssim C(u_0,u_1)$.
\end{thm}

%-------------------------------------------------------------------------

\subsection{Heuristic overview of the proof}

A naive approach to proving Theorem \ref{main_thm_hyp} would be to attempt
uniform control in time of the Sobolev norm $\lp{u(\rho)}{H^1_y}$. We see from equation 
\eqref{KG_hyp} this is almost possible if $\alpha_0=0$ and $\beta$ is constant. By Duhamel's
principle the contribution of the non-linear terms to an $L^\infty_\rho(H^1_y)$ bound for the solution
would be $\int_1^T \rho^{-1} \lp{u(\rho)}{H_y^1}^3 d\rho$, which just fails.
A closer inspection shows  this failure is really only a matter of relatively low frequencies
because one only  needs $H^{\frac{1}{2}+\epsilon}_y$ to control $u$ in $L^\infty_y$. Thus, the only 
significant contribution to the Duhamel integral is $\int_1^T \rho^{-1}
 \lp{P_{<c\ln(\rho)}u(\rho)}{H_y^1}^3 d\rho$ where $P_{<c\ln(\rho)}$ denotes projection onto
 spatial frequencies $|\xi| \lesssim \rho^c$.
 To handle this low frequency logarithmic divergence  
  one needs to directly manipulate the equation. By frequency
projecting the evolution one finds:
\begin{equation}
		(\Box_\mathcal{H}+1)P_{<c\ln(\rho)}u \ = \ \rho^{-1}\beta (P_{<c\ln(\rho)}u)^3 +
		\hbox{\{Better Terms\}} \ , 
\end{equation}
where the ``better term'' on the RHS are $O(\rho^{-1-\frac{1}{10}c})$ and hence integrable. 
For small $c$ the equation on the LHS above is essentially an ODE in time 
(thanks to $\rho^{-2}\partial_y^2$), so one may obtain $L^\infty$ estimates directly for
$P_{<c\ln(\rho)}u$ by simply using the almost Hamiltonian structure of $\ddot u + u=\rho^{-1}\beta u^3$.
One just has to deal with small (integrable) errors when bounding $\frac{1}{2}(\dot{u}^2+u^2)(\rho)$.
Combining this $L^\infty_y$ control for low frequencies with a slow growth estimate of the form 
$\lp{u(\rho)}{H^1_y}\lesssim \epsilon \rho^\delta$
guarantees uniform $L^\infty_y$ bounds for  all of $u$,  and this   suffices to stabilize slow $ H^1_y$
growth  via Duhamel's principle and a bootstrapping argument.

Now up the ante to considering \eqref{KG_hyp} with $\alpha_0\neq 0$ but $\beta$ still constant.
The only difference here is that we need to employ a suitable version of Shatah normal forms
to remove the quadratic term. In doing so one produces additional cubic terms of the type 
$\rho^{-1}\alpha_0^2 u\dot{u}^2$, but these are harmless in the nonlinear frequency localized
$L^\infty_y$ estimates of the previous paragraph because one still retains the needed Hamiltonian
structure. The implementation of Shatah normal forms in hyperbolic coordinates is more or less straight
forward, and moreover
has a nice semi-classical flavor due the the rescaling of the $\partial_y$ derivatives in time.

Finally we come to the full problem considered here, which is to allow a local non-constant perturbation
of $\beta$ with respect to the spatial variable. At first it would appear that the method sketched in the
previous paragraphs breaks down completely because one has $x\approx \rho y$ in hyperbolic
coordinates, so differentiation  of $\beta(x)$ by $\partial_y$ leads to very badly diverging terms.
However, after some frequency localizations one can show that the only troublesome 
contribution stems from the interactions $\beta_{high}u_{low}^3$. Asymptotically one can 
replace the low frequency factors with ``plane waves'' $u_{low}\approx a_+e^{ i\rho} + a_-e^{-i\rho}$.
This leads to an  ansatz for the worst contribution from the non-linearity which is roughly
$\mathcal{N}_{cubic}=
(\Box_\mathcal{H}+1)^{-1}(c_\pm \rho^{-1}\beta_{high} e^{\pm i\rho} + 
d_\pm \rho^{-1}\beta_{high} e^{\pm 3 i\rho})$
for various coefficient $c_\pm,d_\pm$.
The nice feature of this equation is that it can be solved more or less explicitly and estimated
via stationary phase calculations. The resulting expression $\mathcal{N}_{cubic}$ can be thought 
of as a 
``cubic normal forms'' correction and is given in terms of trilinear paraproducts applied to $u$.
See formula \eqref{cubic_NF_ansatz} below. This expression carries all the relevant information 
about the resonant interaction between the spacial frequencies of $\beta$ and the time 
oscillations  of the solution $u$. With a little bit of work one can show that the remaining quantity
$w=u-\mathcal{N}_{cubic}$ is amenable to the direct $H^1_y$ energy and non-linear $L^\infty$
estimates mentioned in the previous paragraphs. This allows one to close the proof.

In order to implement the strategy sketched above, we introduce a number of function spaces
in the next Section. The $N$ space is used to bound source terms, and is simply a Duhamel type norm
that allows one to conclude $L^\infty_\rho(H^1_y)$ estimates. We also introduce two solution
spaces $S$ and $\dot{S}$. Since our $L^\infty_\rho(H^1_y)$ estimates necessarily grow in time
(again, due to low frequencies), one must add a sharp (non-growing) $L^\infty_y$ component
to the solution norms.
This latter part cannot be recovered directly in terms of the $N$ norm (at least for 
low frequencies), but is instead
bounded (for low frequencies) by directly integrating the equation as explained above. The    
$\dot{S}$ norm is slightly weaker than the $S$ norm and is used to control $\dot u$. Since this quantity occurs
so many times when computing normal forms expressions, it deserves to have its own function space
which helps to streamline  many of the product estimates which occur in the sequel.

%-------------------------------------------------------------------------

\subsection{A preliminary coefficient estimate}

We end this section with an elementary lemma that lets us write the
nonlinearity of \eqref{KG_hyp} in a more useful form:

\begin{lem}[Coefficient approximation]\label{rhoy_lem}
Let $\varphi\in \mathcal{S}$. Then if we set $F(\rho,y)= \varphi(\rho y)- \varphi(\rho\sinh(y)) $,
one has the estimate:
\begin{equation}
		\lp{\partial_\rho^n (\rho^{-1}\partial_y)^m F}{L^p_y} \ \leqslant  \ 
		C_{n,m}(\lp{\varphi}{\mathcal{S}}) \rho^{-2-\frac{1}{p}} \ . \label{coeff_approx_est}
\end{equation}
\end{lem}

\begin{proof}[Proof of lemma]
The proof is by induction on the number of derivatives. When $n=m=0$
we have the bound $|F(\rho,y)|\lesssim \rho^{-2}\langle \rho y\rangle^{-2}$
which is easy when $|y|\geqslant \epsilon$,  and for $|y|<\epsilon$
follows from the mean value
theorem which gives:
\begin{equation}
		F(\rho,y) \ =\ \varphi'\big(\xi(\rho,y)\big)O(\rho y^3) \ , 
		\qquad \hbox{where \ \ } |\xi(\rho,y)|\geqslant \rho |y| \ . \notag
\end{equation}
 Taking first derivatives we have:
\begin{equation}
		\partial_\rho F(\rho, y) \ = \ \rho^{-1}
		\td{F}(\rho, y) \ , \qquad \td{F}(\rho, y) \ = \ \td\varphi(\rho y) - 
		\td\varphi(\rho \sinh(y)) \ , \notag 
\end{equation}
where $\td\varphi(x)=x\varphi'(x)$. And:
\begin{equation}
		\rho^{-1}\partial_y  F(\rho, y) \ = \ 
		\big( \varphi'(\rho y) 
		-  \varphi'(\rho \sinh(y))\big) - G(\rho,y) \ , 
		\qquad G(\rho,y) \ = \ \rho^{-2}(1+\cosh(y))^{-1} 
		\td{\td{\varphi}}(\rho \sinh(y)) \ , \notag
\end{equation}
where $\td{\td{\varphi}}(x)=x^2\varphi'(x)$. Then the proof concludes by induction and 
the easy estimate:
\begin{equation}
		\lp{\partial_\rho^n (\rho^{-1}\partial_y)^m G }{L^p} \ \leqslant \ C_{n,m}
		(\lp{\td{\td{\varphi}}}{\mathcal{S}}) 
		\rho^{-2-\frac{1}{p}} \ , \notag
\end{equation}
for $G$ of the form above.
\end{proof}

%-------------------------------------------------------------------------
%%%%%%%%%%%%%%%%%%%%%%%%%%%%%% 
%-------------------------------------------------------------------------

\section{Norms}\label{Norms_sect}

The main norms we work with in the paper are as follows.

\begin{defn}[Function spaces]
For sufficiently smooth and well localized functions we define the solution and source spaces,
where we use the shorthand $\dot u=\partial_\rho u$:
\begin{align}
		\lp{u}{S[1,T]} \! &= \! \sup_{1\leqslant \rho \leqslant T} \lp{u(\rho)}{S(\rho)}
		\ , &\hbox{where \ \ }&
		\lp{u(\rho)}{S(\rho)} \! = \! 
		\rho^{-\delta} \lp{(u,\dot{u},\rho^{-1}\partial_y u)(\rho)}{H^1_y} +  
		\lp{(u,\dot u)(\rho)}{L^\infty_y \cap L^2_y} \ , \notag\\
		\lp{F}{N[1,T]} \! &= \! \sup_{1\leqslant \rho \leqslant T} \lp{F(\rho)}{N(\rho)}
		\ , &\hbox{where \ \ }&
		\lp{F(\rho)}{N(\rho)} \! = \! 
		\rho^{1-\delta} \lp{F(\rho)}{H^1_y} + \rho \lp{F(\rho)}{L^\infty_y\cap L^2_y } \ . \notag
\end{align}
We also introduce an auxiliary space to hold time derivatives of solutions to the Klein-Gordon
equation:
\begin{equation}
		\lp{u(\rho)}{\dot{S}(\rho)} \ = \ \rho^{-\delta} \lp{u}{H^1_y} +\lp{u}{L^\infty_y\cap L^2_y} +
		\inf_{v\in S(\rho)}\big(
		\rho^{-\delta}\lp{\dot u - \rho^{-2}\partial_y^2 v}{H^1_y} +
		\lp{\dot u - \rho^{-2}\partial_y^2 v}{L^\infty_y\cap L^2_y}  + \lp{v}{S(\rho)}
		\big) \ . \notag
\end{equation}
Then define $\lp{u}{\dot S[1,T]}=\sup_{1\leqslant \rho\leqslant T}\lp{u(\rho)}{\dot S(\rho)}$ as usual.
\end{defn}

These norms have the obvious uniform inclusions:
\begin{equation}
		S(\rho) \ \subseteq \ \dot S(\rho) \ , \qquad
		\rho ^{-1}  \dot S(\rho) \ \subseteq \  N(\rho) \ , \qquad
		\rho ^{-1} \partial_\rho S(\rho) \ \subseteq\  N(\rho) \ ,  
		 \label{easy_S_to_N}
\end{equation}
which also directly implies the corresponding embeddings for $S[1,T]$,
$\dot S[1,T]$, and $N[1,T]$.
There are a number of other basic estimates for  norms that will come up later,
so we collect them here.

\begin{lem}[Basic properties of the norms]\label{basic_norm_lemma}
The norms defined above have the following algebraic properties:
\begin{align}
		\lp{FG}{N(\rho)} 
		\ &\lesssim \  \rho^{-1}\lp{F}{N(\rho)}\lp{G}{N(\rho)}  \ ,   
		&\lp{(u,\dot u) F}{N(\rho)} 
		\ &\lesssim \   \lp{u}{S(\rho)}\lp{F}{N(\rho)}  \ ,  \label{alg1}\\
		\lp{uv}{S(\rho)} \ &\lesssim \ \lp{u}{S(\rho)} \lp{v}{S(\rho)} \ , 
		&\lp{uv}{\dot S(\rho)} \ &\lesssim \ \lp{u}{S(\rho)}\lp{v}{\dot S(\rho)} \ . \label{alg2}
\end{align}
In addition, the $S(\rho)$ and $\dot S(\rho)$ norms are related as follows:
\begin{equation}
		\lp{(1-\rho^{-2}\partial_y^2)^{-\frac{1}{2}} u}{S(\rho)} \! \lesssim \! \lp{u}{\dot{S}(\rho)}
		\ , \qquad 
		\lp{\dot u}{\dot{S}(\rho)} \! \lesssim \! \lp{u}{S(\rho)} + \rho^{-\delta}
		\lp{(\Box_\mathcal{H}+1)u}{H^1_y} + \lp{(\Box_\mathcal{H}+1)u}{L^2_y\cap L^\infty_y}
		\ . \label{dot_S_bnds}
\end{equation}
Finally,  one has the Duhamel type estimate:
\begin{equation}
		\sup_{1\leqslant \rho \leqslant T}
		\rho^{-\delta} \lp{(u,\dot{u},\rho^{-1}\partial_y u)(\rho)}{H^1_y} \ \lesssim \
		\lp{(u_0,u_1)}{H^2_y\times H^1_y} + \delta^{-1}\lp{F}{N[1,T]} \ , \label{duhamel}
\end{equation}
for solutions to the problem $(\Box_\mathcal{H}+1)u  =  F$ with 
$u(1) =  u_0$ and $ \partial_\rho u(1)  =  u_1$.
\end{lem}

The proofs of the above estimates are for the most part immediate, with the exception
of the second bound on line \eqref{alg2} and the first bound
on line \eqref{dot_S_bnds}. These will be shown   in Section \ref{Harm_sect}.

\begin{rem}
Note that the embedding $(\Box_\mathcal{H}+1)^{-1}N[1,T] \hookrightarrow S[1,T]$
is not uniform and grows like $T^\delta$. To overcome this obstacle we will need 
to use nonlinear estimates to recover the $L^\infty_y\cap L^2_y$ norm of $u$ at low frequencies. 
\end{rem}

Finally, we list here a nonlinear notation for norms. We let $\mathcal{Q}_k(a,b,c,\ldots )$ 
denote any polynomial which has a term of  at least degree $k$ in the first entry
as a factor in each monomial.
A typical example would be the function $\mathcal{Q}_2(a,b)=(a^2 + a^4)(1 +b^2)$ which would then
be used as:
\begin{equation}
		\mathcal{Q}_2(\lp{u}{S[1,T]}, \lp{F}{S[1,T]}) \ =\ 
		(\lp{u}{S[1,T]}^2 + \lp{u}{S[1,T]}^4 )(1+\lp{F}{S[1,T]}^2) \ ,
\end{equation}
where $F$ is some auxiliary quantity appearing in an estimate. The exact form of $\mathcal{Q}$
may change from line to line. In one place in the sequel we'll also use 
$\mathcal{Q}_\frac{3}{2}(\lp{u}{S[1,T]})$ with the obvious interpretation.

%-------------------------------------------------------------------------
%%%%%%%%%%%%%%%%%%%%%%%%%%%%%% 
%-------------------------------------------------------------------------

\section{Main Constructions and Estimates}\label{Main_sect}

The main construction of this paper is the following:

\begin{thm}[Normal Form Corrections]\label{nf_thm}
Let $\alpha_0,\beta_0\in \mathbb{R}$ and $\beta_1=\beta_1(\rho y)$
for some $\beta_1(x)\in \mathcal{S}$, and set $\beta=\beta_0+\beta_1$.
Let $u$ and $w$ solve the equations:
\begin{equation}
		(\Box_\mathcal{H} + 1)u \ = \  \rho^{-\frac{1}{2}}\alpha_0 u^2 + \rho^{-1}\beta u^3 + F
		\ , \qquad
		(\Box_\mathcal{H} + 1)w \ = \     \rho^{-1}\beta_1 w^3 + G
		\ , 
		\label{hyp_KG}
\end{equation}
on the time interval $[1,T]$. Then there exists   nonlinear expressions 
$\mathcal{N}_{quad}=\mathcal{N}_{quad}(u,\dot u)$ and 
$\mathcal{R}_{quad}=\mathcal{R}_{quad}(u,\dot u,F,\dot{F})$
and $\td{\mathcal{R}}_{quad}=\td{\mathcal{R}}_{quad}(u,\dot u,F,\dot{F})$,
 $\mathcal{N}_{cubic}=\mathcal{N}_{cubic}(w,\dot w)$ and 
$\mathcal{R}_{cubic}=\mathcal{R}_{cubic}(w,\dot w,G,\dot G)$,
such that one has the algebraic identities:
\begin{align}
		(\Box_\mathcal{H} + 1)\big( u - \mathcal{N}_{quad}\big) \ &= \ 
		\rho^{-1}\beta  u^3  + F + \mathcal{R}_{quad} \ = \ 
		\rho^{-1}(\beta + 2\alpha_0^2) u^3 -\frac{8}{3} \rho^{-1}\alpha_0^2
		\dot{u}^2 u + F + \td{\mathcal{R}}_{quad} \ , \label{main_nf_iden_quad}\\
		(\Box_\mathcal{H} + 1)\big( w - \mathcal{N}_{cubic}\big) \ &= \ 
		G + \mathcal{R}_{cubic} \ . \label{main_nf_iden_cubic}
\end{align}
For the quadratic corrections one has the   estimates: 
\begin{multline}
		\lp{\rho^\frac{1}{2} \mathcal{N}_{quad} }{S[1,T]} + \lp{(\mathcal{R}_{quad},
		\dot{\mathcal{R}}_{quad})}{N[1,T]}
		+ \lp{  \td{\mathcal{R}}_{quad}
		 }{ L^1_\rho ( L^\infty_y ) [1,T]}
		 \ \lesssim \ \lp{\rho^\frac{1}{2} F}{S[1,T]}^2 +
		  \lp{\rho^\frac{1}{2} \dot{F}}{\dot S[1,T]}^2\\
		 +\lp{\rho( F,\dot F)}{L^\infty_{\rho,y}[1,T]}^2
		+ \mathcal{Q}_2\big(\lp{u}{S[1,T]}, \lp{\rho^\frac{1}{2} F}{S[1,T]} ,
		  \lp{\rho^\frac{1}{2} \dot{F}}{\dot S[1,T]}, \lp{\rho( F,\dot F)}{L^\infty_{\rho,y}[1,T]} \big) 
		 \ . \label{main_nf_ests_quad}
\end{multline}
In the case of the cubic corrections one has:
\begin{equation}
		\lp{\mathcal{N}_{cubic} }{S[1,T]}  		
		+ \lp{\mathcal{R}_{cubic} }{N[1,T]}\ \lesssim \ 
		\mathcal{Q}_2\big(\lp{w}{S[1,T]}, \lp{G}{N[1,T]},\lp{\rho \dot G}{L^\infty_\rho(L^2_y)[1,T]}\big)
		 + \mathcal{Q}_3\big( \lp{G}{N[1,T]}\big)
		  \ . \label{main_nf_ests_cubic}
\end{equation}
\end{thm}

To this information we add the following nonlinear $N\to S$ space bound:

\begin{prop}[Nonlinear Estimates]\label{nonlin_est_prop}
Let $a=a_0+ a_1(\rho y)$ and $b=b_0 + b_1(\rho y)$ for real functions 
$a,b$ with $a_1,b_1\in \mathcal{S}$ and $a_0,b_0\in\mathbb{R}$.  
Let $u,w$ solve the inhomogeneous nonlinear equations:
\begin{align}
		(\Box_\mathcal{H}+1 + \rho^{-\frac{1}{2}} a u + \rho^{-1}b u^2 )u \ &= \ F \ , \label{nonlin1}\\
		(\Box_\mathcal{H}+1 + \rho^{-1} a w^2 + \rho^{-1}b \dot{w}^2 )w \ &= \ G \ , \label{nonlin2}
\end{align}
on the interval $[1,T]$, with initial data $u(1)=u_0$ and $\partial_\rho u(1)=u_1$
and similarly for $w$.
Then one has the   nonlinear estimates:
\begin{align}
		\lp{(u,\dot u)}{L^\infty_\rho( L^2_y) [1,T]  }  &\lesssim \lp{(u_0,u_1)}{H^1_y \times L^2_y}+  
		\mathcal{Q}_\frac{3}{2} \big(\lp{u}{S[1,T]} \big) +  
		\lp{F}{L^1_\rho (L^2_y )  [1,T]} \ . \label{nonlin_nrg_est1}\\
		\lp{(w,\dot w)}{  L^\infty_{\rho,y} [1,T]} &\lesssim  \lp{(w_0,w_1)}{H^2_y \times H^1_y} 
		+ \!\!\!\!\sup_{1\leqslant \rho\leqslant T} \rho^{-c}\lp{w(\rho)}{S(\rho)}
		+ \mathcal{Q}_2\big(\lp{w}{S[1,T]} \big)+
		\lp{G}{L^1_\rho (  L^\infty_y)  [1,T]} \ . \label{nonlin_nrg_est2}
\end{align}
for some sufficiently small $c>0$ in the second estimate.
\end{prop}

Combining Theorem \ref{nf_thm} and Proposition \ref{nonlin_est_prop} 
we have the following a-priori bound which implies  Theorem 
\ref{main_thm_hyp}
through the usual local existence and continuity argument.

\begin{thm}[Main A-Priori Estimate]\label{main_ap_thm}
Let $\beta=\beta_0+\beta_1$ with $\beta_1\in \mathcal{S}$, and let $u$ solve the equation:
\begin{equation}
		(\Box_\mathcal{H}+1)u \ = \ \rho^{-\frac{1}{2}}\alpha_0 u^2 + 
		\rho^{-1} \beta\big(\rho\sinh(y)\big)u^3 \ , \notag
\end{equation}
on $[1,T]$ with initial data 
$u(1)=u_0$ and $\partial_\rho u(1)=u_1$. Then one has the following a-priori
estimate for some $C\geqslant 2$:
\begin{equation}
		\lp{u}{S[1,T]} \ \lesssim \ \lp{(u_0,u_1)}{H^2_y \times H^1_y} 
		+ \lp{u}{S[1,T]}^\frac{3}{2} + \lp{u}{S[1,T]}^{C} \ 
		. \label{main_ap_est}
\end{equation}
\end{thm}

\begin{proof}[Proof that Theorem \ref{nf_thm} and Proposition \ref{nonlin_est_prop} imply
Theorem \ref{main_ap_thm}]
The proof is in a series of steps.

\step{1}{Coefficient reduction and the $L^2_y$ estimate}
First we write:
\begin{equation}
		(\Box_\mathcal{H}+1)u \ = \ \rho^{-\frac{1}{2}}\alpha_0 u^2 + 
		\rho^{-1} \beta(\rho y)  u^3 + F \ ,
		\qquad \hbox{where\ \ } 
		F\ = \ \rho^{-1}\big( \beta_1(\rho \sinh(y))-\beta_1(\rho y)\big)u^3 \ . \notag
\end{equation}
For the functions $F(\rho,y,u)$ one has from Lemma \ref{rhoy_lem} and
the various product estimates and inclusions in Lemma \ref{basic_norm_lemma}
the collection of bounds :
\begin{equation}
		\lp{F}{L^1_\rho (L^2_y \cap L^\infty_y)  [1,T]} \!+\!
		\lp{\rho^\frac{1}{2} F}{S[1,T]} \!+\!
		  \lp{\rho^\frac{1}{2} \dot{F}}{\dot S[1,T]}
		 \!+\! \lp{\rho( F,\dot F)}{L^\infty_{\rho,y}[1,T]} 
		 \!+\! \lp{(F,\dot F)}{N[1,T]}
		  \! \lesssim\!  \mathcal{Q}_3\big(
		 \lp{u}{S[1,T]}
		 \big) \ . \label{F_bounds}
\end{equation}
By combining the $L^1L^2$ estimate for $F$   above with the nonlinear 
energy estimate \eqref{nonlin_nrg_est1} we have:
\begin{equation}
		\lp{(u,\dot u)}{L^\infty_\rho(L^2_y )[0,T]} \ \lesssim \ \lp{(u_0,u_1)}{H^1_y \times L^2_y}+  
		\mathcal{Q}_\frac{3}{2}\big( \lp{u}{S[1,T]} \big)  \ . \label{L2_part}
\end{equation}

\step{2}{Quadratic correction and the $L^\infty_y$ estimate}
Next, introduce $w=u-\mathcal{N}_{quad}$ 
which solves the equation:
\begin{equation}
		(\Box_\mathcal{H}+1)w \ = \ \rho^{-1}(\beta + 2\alpha_0^2) w^3 -\frac{8}{3} \rho^{-1}\alpha_0^2
		\dot{w}^2 w + \td{F} \ , \notag
\end{equation}
where:
\begin{equation}
		\td{F} \ = \ \rho^{-1}(\beta + 2\alpha_0^2) (u^3-w^3) -\frac{8}{3} \rho^{-1}\alpha_0^2(
		\dot{u}^2 u-\dot{w}^2 w) + F + \td{\mathcal{R}}_{quad} \ , \notag
\end{equation}
and where $\td{\mathcal{R}}_{quad}$ is defined on the far RHS of line \eqref{main_nf_iden_quad} 
and satisfies the relevant portion of estimate \eqref{main_nf_ests_quad}.
Combining   with the improved $S$ norm
bounds for $\mathcal{N}_{quad}=u-w$ from line \eqref{main_nf_ests_quad}, and using the $L^1L^\infty$ 
part of \eqref{F_bounds} as well,
we have $\lp{\td{F}}{L^1_\rho(L^\infty_y)[1,T]}  \lesssim 
\mathcal{Q}_2(\lp{u}{S[1,T]})$. Therefore by estimate \eqref{nonlin_nrg_est2} we control
$\lp{(w,\dot w)}{L^\infty_{\rho,y}[1,T]}$,
and by a further application of the  $S$ estimate  \eqref{main_nf_ests_quad}
for $\mathcal{N}_{quad}$ (which also gives an estimate for the data $\lp{(w_0,w_1)}{H^2_y \times H^1_y}$
thanks to $S(1)\subseteq H^2_y \times H^1_y$) 
we have:
\begin{equation}
		\lp{(u,\dot u)}{  L^\infty_{\rho,y}[1,T] } \ \lesssim \ 
		\sup_{1\leqslant \rho\leqslant T} \rho^{-c}\lp{u(\rho)}{S(\rho)}+
		\lp{(u_0,u_1)}{H^2_y \times H^1_y} 
		+ \mathcal{Q}_2(\lp{u}{S[1,T]}) \ . \label{L00_part}
\end{equation}

\step{3}{Cubic correction and the $\rho^\delta H^1_y$ estimate}
Finally, we rewrite the equation for $w$ above as:
\begin{equation}
		(\Box_\mathcal{H}+1)w \ = \ \rho^{-1} \beta_1   w^3 +G \ , \notag
\end{equation}
where:
\begin{equation}
		G \ = \ \rho^{-1} \beta_1   (u^3- w^3)+
		\rho^{-1} \beta_0   u^3  + F + \mathcal{R}_{quad} \ . \notag
\end{equation}
In order to invoke the cubic part of Theorem \ref{nf_thm} we need to show the estimate:
\begin{equation}
		\lp{G}{N[1,T]} + \lp{\rho \partial_\rho G}{L^\infty_\rho(L^2_y)[1,T]}
		\ \lesssim \ \mathcal{Q}_2(\lp{u}{S[1,T]}) \ . \label{G_est}
\end{equation}
For the quantities $F$ and $\mathcal{R}_{quad}$ we use  \eqref{F_bounds}
and \eqref{main_nf_ests_quad}. For the $N[1,T]$ norm of  the remaining cubic expressions
in $u$ and $w$ we use the  inclusions on line
\eqref{easy_S_to_N} and the second algebra estimate from line \eqref{alg1}, as well
as estimate \eqref{main_nf_ests_quad} 
for $\mathcal{N}_{quad}=u-w$ combined with $\rho^{-\frac{1}{2}}\beta_1\in S(\rho)$.
The $L^\infty L^2$ bound for $\rho\partial_\rho$ applied to the cubic terms in $G$ is immediate from 
the $S$ norm bounds for $u$ and $w$.

Now form the quantity $v=w-\mathcal{N}_{cubic}$
which satisfies equation \eqref{main_nf_iden_cubic} and estimates 
\eqref{main_nf_ests_cubic}. Then by applying \eqref{duhamel} to the equation
for $v$ and  using estimates \eqref{main_nf_ests_cubic} for $\mathcal{R}_{cubic}$
and \eqref{G_est} for $G$
in the resulting RHS, we control 
$\lp{\rho^{-\delta}(v,\dot v, \rho^{-1}\partial_y v)(\rho)}{L^\infty_\rho( H^1_y)  }$.
Finally, we can switch back to the original $u$ via 
the estimates for $\mathcal{N}_{cubic}$ on line \eqref{main_nf_ests_cubic}
and the estimates for $\mathcal{N}_{quad}$ on line \eqref{main_nf_ests_quad}. All together this gives:
\begin{equation}
		\lp{\rho^{-\delta}(u,\dot u, \rho^{-1}\partial_y u)(\rho)}{L^\infty_\rho( H^1_y) [1,T] } \ 
		\lesssim \ \lp{(u_0,u_1)}{H^2_y \times H^1_y}+  
		\mathcal{Q}_2\big( \lp{u}{S[1,T]} \big)  \ . \label{main_H1_bound}
\end{equation}
By adding this last bound to \eqref{L2_part} and \eqref{L00_part} above we have \eqref{main_ap_est}
for some $C\geqslant 2$.

\step{4}{Combination}
By adding this last bound to \eqref{L2_part} and \eqref{L00_part} above, we have
for $T_0>1$ sufficiently large the following estimate uniform in $T\geqslant 1$:  
\begin{equation}
		\lp{u}{S[1,T]} \ \lesssim \ 
		 \lp{u}{S[1,T_0]}+
		\lp{(u_0,u_1)}{H^2_y \times H^1_y} 
		+ \mathcal{Q}_\frac{3}{2}( \lp{u}{S[1,T]})  \ . \label{almost_there}
\end{equation}
On the other hand one from \eqref{main_H1_bound} and Sobolev embedding we easily
have:
\begin{equation}
		\lp{u}{S [1,T_0] } \ 
		\lesssim \ T_0^\delta \big( \lp{(u_0,u_1)}{H^2_y \times H^1_y}+  
		\mathcal{Q}_2( \lp{u}{S[1,T]} ) \big) \ . \notag
\end{equation}
Substituting this for the first term on RHS \eqref{almost_there} yields \eqref{main_ap_est}
for some sufficiently large $C\geqslant 2$.

\end{proof}

%-------------------------------------------------------------------------

%\subsection{Proof of the Main A-priori Estimate}

%-------------------------------------------------------------------------
%%%%%%%%%%%%%%%%%%%%%%%%%%%%%% 
%-------------------------------------------------------------------------

\section{Fourier Analytic Setup}\label{Harm_sect}

We assume the reader is familiar with basic Littlewood-Paley theory, which we only
employ in the spatial variable.
We let $P_k$ be smooth cutoff on frequency $\xi\sim 2^k$, and set $P_{<k}=\sum_{j<k} P_j$.
We  use $u_k=P_k u$ and $u_{<k}=P_{<k}u$ etc. Recall  
that  $P_k$ and $P_{<k}$ are uniformly bounded on all $L^p$ spaces. 
As a consequence it is easy to check that $P_k$ and $P_{<k}$ are bounded with respect to 
$S$, $\dot S$, and $N$ norms.

In the $(\rho,y)$ variables there is no action at low frequencies for the evolution of \eqref{eqs}, 
so we use the convention that $P_0=P_{<0}$
whenever we are in these variables. When in the $(\rho,y)$ variables we'll
use subscripts such as $L^p_y$, $H^1_y$ to remind the reader of this.
However, in some places we rescale
back to $x$ variables via $x=\rho y$, in which case we let $P_k$ denote
cutoff on $\xi\sim 2^k$ for all $k\in \mathbb{Z}$. 

Several basic things that will 
be used freely are:
\begin{equation}
		\lp{P_k u }{L^p(\mathbb{R})} \ \lesssim \ 2^{\frac{1}{q}-\frac{1}{p}}\lp{P_k u}{L^q(\mathbb{R})}
		\ , \ \  
		\hbox{and\ \ } \lp{\partial_x P_k u }{L^p(\mathbb{R})} \ \approx \ 2^k 
		\lp{P_k u}{L^p(\mathbb{R})} \ ,
		\qquad \hbox{for all\ \ } 1\leqslant q\leqslant p\leqslant \infty \ . \notag
\end{equation}
A direct consequence of these is the following bound that will be used many times 
in the sequel:
\begin{equation}
		2^{(k-\ln_2(\rho))_+}\lp{P_k u}{L^\infty_y} 
		\!+\! \lp{P_k \dot u}{L^\infty_y}
		\! \lesssim \! \rho^{\delta}2^{-\frac{1}{2}k}
		\lp{u}{S(\rho)}
		\ , \ \ \hbox{and \ \ }
		\lp{P_{ k} 
		u}{L^\infty_y} \! \lesssim \! \rho^{\delta}2^{-\frac{1}{2}k}\lp{u}{\dot S(\rho)}
		\ . \label{L00_sob}
\end{equation}
In particular notice that this  gives $\lp{\rho^{-1}\partial_y u}{L^\infty_y}\lesssim 
\lp{u}{S(\rho)}$, which  will be used freely in the sequel.

Finally, we  make the special notational convention that $D_y=\frac{1}{i\rho}\partial_y$,
which can be thought of as a semi-classical derivative where $h=\rho^{-1}$. 
However, note this assignment is time dependent so one has the 
identity $[\partial_\rho, D_y]=-\rho^{-1}D_y$.
We set $D_x=\frac{1}{i}\partial_x$, $D_\rho=\frac{1}{i}\partial_\rho$, etc as usual.

\subsection{Multilinear Estimates Pt. I: Paraproducts}

We list here a number of elementary product estimates for frequency localized functions.
Our purpose is to set these up in a notational way that will be convenient for later use.
The details are  standard.

\begin{lem}[Basic high-low product estimate]\label{basic_prod_lemma}
Let $k\in\mathbb{Z}$. For test functions $f_k, u^{(i)}_{<k+O(1)}$ and $s_i \in\mathbb{N}$ one has:
\begin{equation}
		\lp{\partial_y^{s_0}
		(f_k \prod_{i=1}^n \partial_y^{s_i}u^{(i)}_{<k+O(1)})
		 }{L^r} \ \lesssim \ \lp{(\partial_y)^{\sum s_i} f_k}{L^{p_0}}
		\prod_{i=1}^n \lp{u^{(i)}_{<k+O(1)}}{L^{p_i}} \ ,
		\quad \hbox{where \ \ } \sum \frac{1}{p_i}= \frac{1}{r}
		\ .  \label{basic_prod_est}
\end{equation}
\end{lem}

From this we derive a number of special cases which will be useful later.

\begin{lem}[High-low $S(\rho)$ and $N(\rho)$ product estimates]\label{used_prod_lem}
Let $k>0$. For test functions $f_k, u^{(i)}_{<k},H_{<k+C}$ and $s_i \in\mathbb{N}$ one has:
\begin{align}
		\lp{\rho^{-1} f_k \prod_{i=1}^n u^{(i)}_{<k}}{S(\rho)}
		\ &\lesssim \ \sum_{j=1}^2
		\lp{(\rho^{-1} f_k, \rho^{-\delta}D_y^j f_k,
		\rho^{-\delta}D_y^{j-1} \dot f_k
		)}{L^\infty_y}\prod_{i=1}^n 
		\lp{u^{(i)}_{<k}}{\dot S(\rho)} \ , \label{SN_mult1}\\
		\lp{\rho^{-1} f_k \prod_{i=1}^n D_y^{s_i} u^{(i)}_{<k}}{N(\rho)}
		\ &\lesssim \ \lp{ (D_y)^{\sum s_i}f_k}{L^\infty_y}\prod_{i=1}^n 
		\lp{u^{(i)}_{<k}}{\td S(\rho)} \ , 
		\qquad \hbox{if \ \ } \sum s_i \geqslant 1 \ , \label{SN_mult2}\\
		\slp{\rho^{-1} f_k H_{<k+C}\!\!\prod_{i=1}^n D_y^{s_i} u^{(i)}_{<k}}{N(\rho)}
		\! &\lesssim \! \rho^{-\delta}
		\slp{\big( \rho^{2\delta-1} (D_y)^{\sum s_i}f_k,
		(D_y)^{1+ \sum s_i  }f_k
		\big) }{L^\infty_y}\! \prod_{i=1}^n 
		\slp{u^{(i)}_{<k}}{\td S(\rho)}\slp{\rho H_{<k+C}}{L^2_y}  . \label{SN_mult3}
\end{align}
In the last two estimates we have set $\td{S}(\rho)=\rho^{\delta}H^1_y\cap L^2_y\cap L^\infty_y$
at fixed time.
\end{lem}

\begin{proof}[Proof of Lemma \ref{basic_prod_lemma}]
By the Leibniz rule and a possible redefinition of the $s_i$  it suffices to consider
the product $\partial_y^{s_0}f_k\cdot \prod_{i=1}^n \partial_y^{s_i}u^{(i)}_{<k+O(1)}$ instead.
By H\"older's inequality we may assume each of the $s_i\geqslant 1$, for $i=1,\ldots,n$.
Breaking the $u^{(i)}_{<k+O(1)}$ into dyadic frequencies and using H\"older's
inequality   we have:
\begin{align}
		\lp{\partial_y^{s_0}f_k\cdot 
		\prod_{i=1}^n \partial_y^{s_i}u^{(i)}_{<k+O(1)})}{L^r} \ &\lesssim \
		\sum_{k_i<k+O(1)}2^{\sum_{i=1}^n s_i k_i}\lp{\partial_y^{s_0}f_k}{L^{p_0}}
		\prod_{i=1}^n\lp{P_{k_i} u^{(i)}_{<k+O(1)}}{L^{p_i}} \ , \notag\\
		&\lesssim \ 2^{(\sum_{i=1}^n s_i)k}  \lp{\partial_y^{s_0} f_k}{L^{p_0}}
		\prod_{i=1}^n \lp{u^{(i)}_{<k+O(1)}}{L^{p_i}} \ . \notag
\end{align}
The proof concludes by trading dyadic weights for derivatives.
\end{proof}

\begin{proof}[Proof of estimate \eqref{SN_mult1}]
First consider the portion of the $S(\rho)$ norm which does not contain $\partial_\rho$ 
derivatives. Then the $L^2_y \cap L_y^\infty$ and $\rho^{\delta}H^1_y$
estimates are all an immediate consequence of \eqref{basic_prod_est} by always
putting $f_k$ and its derivatives in $L^\infty_y$, one of the $u^{(i)}_{<k}$ in either $L^\infty_y$ or $L^2_y$,
and the rest of the $u^{(i)}_{<k}$ in $L^\infty_y$.

Now consider the  portion of  $S(\rho)$ containing $\partial_\rho$
derivatives. If $\partial_\rho$
lands on $f_k$ then the proof is similar to what was just discussed.
If the $\partial_\rho$ lands on a $u^{(i)}_{<k}$ assume WLOG it is
$u^{(1)}_{<k}$. To use the structure of the $\dot S$ space
we write $\dot{u}^{(1)}_{<k}=P_{<k+C}(\dot{u}^{(1)}_{<k} +D_y^2 v)- D_y^2 v_{<k+C}=w_{<k+C}
- D_y^2 v_{<k+C}$ where $\lp{w_{<k+C}}{\rho^\delta H^1_y\cap L^\infty_y \cap L^2_y}
+ \lp{v_{<k+C}}{S(\rho)}\lesssim \lp{u^{(1)}_{<k}}{\dot{S}(\rho)}$.
Therefore, modulo bounds which are again of the form just discussed we need to prove:
\begin{equation}
		\lp{\rho^{-1}f_k D_y^2 v_{<k+C}\prod_{i=2}^n u^{(i)}_{<k}}
		{\rho^\delta H^1_y\cap L^\infty_y \cap L^2_y} \! \lesssim \! 
		 \rho^{-\delta} \lp{D_y f_k}{L^\infty_y}\lp{v_{<k+C}}{S(\rho)}
		 \prod_{i=2}^n \lp{u^{(i)}_{<k}}{\dot S(\rho)}
		\ . \notag
\end{equation}
In all cases the desired estimate again follows from \eqref{basic_prod_est}
by putting $D_y f_k$ in $L^\infty_y$, $D_y v_{<k+C}$ in $L^\infty_y$ or $L^2_y$,
$D_y^2 v_{<k+C}$ in $L^2_y$,
and the rest of the $u^{(i)}_{<k}$ in $L^\infty_y$.
\end{proof}

\begin{proof}[Proof of estimate \eqref{SN_mult2}]
The estimate for the $L^2_y\cap L^\infty_y$ part of the $N(\rho)$ norm follows at once from
\eqref{basic_prod_est}. In the case of the $\dot H^1_y$ portion of the norm we 
need to use the condition that $\sum s_i\geqslant 1$. WLOG assume that 
$s_1\geqslant 1$. Then from \eqref{basic_prod_est} we have:
\begin{equation}
		\rho^{-\delta}
		\lp{ \partial_y (  f_k \prod_{i=1}^n D_y^{s_i} u^{(i)}_{<k} ) }{L^2_y}
		\ \lesssim \ \lp{ (D_y)^{\sum s_i}f_k}{L^\infty_y}
		\rho^{-\delta}\lp{\partial_y u^{(1)}_{<k}}{L^2_y}
		\prod_{i=2}^n 
		\lp{u^{(i)}_{<k}}{L^\infty_y} \ , \notag
\end{equation}
which is sufficient to achieve RHS \eqref{SN_mult2}.  
\end{proof}

\begin{proof}[Proof of estimate \eqref{SN_mult3}]
For the $L^2_y$ and $\rho^\delta H^1_y$ portions of the $N(\rho)$ norms this
again boils down to simple manipulations involving \eqref{basic_prod_est} and
the details are left to the reader. 

To prove \eqref{SN_mult3} for the $L^\infty$ portion of $N(\rho)$ we combine
\eqref{basic_prod_est} with a Sobolev embedding which gives:
\begin{equation}
		\lp{ f_k H_{<k+C}\prod_{i=1}^n D_y^{s_i} u^{(i)}_{<k}}{L^\infty_y}
		\ \lesssim \ 2^{\frac{1}{2}k}\rho^{-1}\lp{(D_y)^{\sum s_i} f_k}{L^\infty_y}
		\prod_{i=1}^n \lp{u^{(i)}_{<k}}{L^\infty_y} \lp{\rho H_{<k+C}}{L^2_y} \ . \notag
\end{equation}
Then RHS above is bounded in terms of RHS \eqref{SN_mult3} by splitting into
cases based on $2^{\frac{1}{2}k}\leqslant \rho^\delta$ and vice versa.
In the latter case we use $2^{\frac{1}{2}k}\rho^{-1}P_k\lesssim \rho^{-\delta} D_y P_k$.

\end{proof}

%--------------------------------------------------------------------------

\subsection{Proof of the Basic $S,\dot S$ Space Properties}

The purpose of this subsection is to prove Lemma \ref{basic_norm_lemma}.
The proofs of both estimates on line \eqref{alg1},  the first estimate
on line \eqref{alg2}, the second estimate on line \eqref{dot_S_bnds}, and 
the estimate on line \eqref{duhamel} are all elementary and left to the 
reader. Here we'll handle the second estimate on line \eqref{alg2}
and the first estimate on line  \eqref{dot_S_bnds}.

\begin{proof}[Proof of the estimate 
$\lp{\langle D_y \rangle^{-1} u}{S(\rho)} \lesssim 
\lp{u}{\dot{S}(\rho)}$]

The main step is to show the operator $\langle D_y\rangle^{-1}$ is bounded on 
all $L^p_y$ spaces. After rescaling this is the same as showing that 
$\langle D_x\rangle^{-1}$ is bounded on all $L^p$ spaces which easily follows by showing
the associated convolution kernel is in $L^1$ (it has only a logarithmic singularity at the
origin and is rapidly decaying). This shows that
$\lp{\langle D_y\rangle^{-1} u }{\rho^\delta H^1_y \cap L^2_y\cap L^\infty_y}
\lesssim \lp{  u }{\rho^\delta H^1_y \cap L^2_y\cap L^\infty_y}$.
The estimate for $D_y\langle D_y \rangle^{-1}u $ in $\rho^{\delta}H^1_y$ 
follows directly from Plancherel.

It remains to bound $\partial_\rho \langle D_y \rangle^{-1} u$ 
in $\rho^\delta H^1_y \cap L^2_y\cap L^\infty_y$. Modulo 
a commutator which produces terms like what we have already bounded in 
the previous paragraph, it suffices to consider $\langle D_y\rangle^{-1}\dot u$.
Choose
a $v$ with $\lp{\dot u + D_y^2 v}{\rho^\delta H^1_y \cap L^2_y\cap L^\infty_y} 
+ \lp{v}{S(\rho)}\lesssim \lp{u}{\dot S(\rho)}$. Then we have:
\begin{equation}
		\lp{\langle D_y \rangle^{-1} \dot u }{\rho^\delta H^1_y \cap 
		L^2_y\cap L^\infty_y} \ \lesssim \ 
		\lp{\langle D_y \rangle^{-1} (\dot u + D_y^2 v)}
		{\rho^\delta H^1_y \cap L^2_y\cap L^\infty_y} + 
		\lp{\langle D_y \rangle^{-1} D_y^2 v}
		{\rho^\delta H^1_y \cap L^2_y\cap L^\infty_y} \ . \notag
\end{equation}
The first term on the RHS above is of the form of the previous paragraph. To handle the 
second we split $v=v_{<\ln_2(\rho)}+ v_{\geqslant \ln_2(\rho)}$. For $v_{<\ln_2(\rho)}$
use $D_y^2 P_{<\ln_s(\rho)}$ is bounded on any $L^p_y$ space. For $v_{\geqslant \ln_2(\rho)}$
we use the  embedding:
\begin{equation}
		P_{\geqslant \ln_2(\rho)}:\, \rho^{\delta} H^1_y \ \longrightarrow \
		L^2_y\cap L^\infty_y  \ , \label{high_embed}
\end{equation}
(which follows from \eqref{L00_sob}), and then the fact that $D_y \langle D_y\rangle^{-1}$
is bounded on $\rho^\delta H^1_y$.
\end{proof}

\begin{proof}[Proof of the estimate $\lp{uv}{\dot S(\rho)} \lesssim  
\lp{u}{S(\rho)}\lp{v}{\dot S(\rho)} $]
The bound for the portion of $\dot S(\rho)$ which does not contain $\dot v$ 
follows directly from the definition of the norms, the Leibniz rule,
and H\"older's inequality. Here we focus attention on showing there exists a 
$w\in S(\rho)$ such that:
\begin{equation}
		\lp{u\dot v + D_y^2 w}{\rho^{\delta} H^1_y\cap L^2_y\cap L^\infty_y}
		+ \lp{w}{S(\rho)}
		\ \lesssim \ \lp{u}{S(\rho)}\lp{v}{\dot S(\rho)} \ . \label{udotv_est}
\end{equation}
Let $\td w\in S(\rho)$ such that 
$\lp{\dot v + D_y^2 \td w}{\rho^{\delta} H^1_y\cap L^2_y\cap L^\infty_y}
+ \lp{\td w}{S(\rho)}\lesssim 
\lp{v}{\dot S(\rho)}$, then set $w=\sum_{k\geqslant \ln_2(\rho)}u_{<k-C}\td{w}_k$.
Then $\lp{w}{S(\rho)}\lesssim \lp{u}{S(\rho)}\lp{\td w}{S(\rho)}$ follows
easily by using the product estimate \eqref{basic_prod_est} and 
orthogonality to gain an $\rho^{\delta }H^1_y$ bound and then using 
\eqref{high_embed}.  
Now decompose the product $u\dot v$ into a number of 
frequency interactions and estimate each term separately.

\case{1}{The  product $u \dot v_{<\ln_2(\rho)}$}
Here we compute:
\begin{equation}
		\lp{u\dot v_{<\ln_2(\rho)}  }{\rho^{\delta} H^1_y\cap L^2_y\cap L^\infty_y}
		\ \lesssim \ \lp{u(\dot v_{<\ln_2(\rho)} 
		+ D_y^2 \td w_{<\ln_2(\rho)})}{\rho^{\delta} H^1_y\cap L^2_y\cap L^\infty_y}
		+ \lp{u D_y^2 \td w_{<\ln_2(\rho)}}{\rho^{\delta} H^1_y\cap L^2_y\cap L^\infty_y}
		 \ , \notag
\end{equation}
and use the fact that $\rho^{\delta} H^1_y\cap L^2_y\cap L^\infty_y$ is an algebra followed
by the boundedness of $D_y^2 P_{<\ln_2(\rho)}$ on $S(\rho)$ to achieve RHS \eqref{udotv_est}.
		
\case{2}{ The product $\sum_{k\geqslant \ln_2(\rho)} u_{<k-C}\dot v_k$}
To bound this use orthogonality of the sum and the embedding \eqref{high_embed},
which reduces matters to bounding the RHS of:
\begin{equation}
		\lp{u_{<k-C}\dot v_k + D_y^2( u_{<k-C}\td w_k)}{H^1_y}  \ \lesssim \ 
		\lp{u_{<k-C}( \dot v_k + D_y^2 \td w_k)}{H^1_y} 
		+ \lp{ D_y^2( u_{<k-C}\td w_k) -u_{<k-C}D_y^2 \td w_k }{H^1_y} \ . \notag
\end{equation}
The first term on the RHS is estimated by the Leibniz rule, while the second uses the fact that
one derivative transfers to the low frequency term. Note that
$\lp{D_y P_{<k-C}u}{L^\infty}\lesssim \lp{u}{S(\rho)}$
follows easily from \eqref{L00_sob}.

\case{3}{The product $\sum_{k\geqslant \ln_2(\rho)} u_{\geqslant k-C}\dot v_k$}
For this term we first use the inequality:
\begin{equation}
		\sum_{k\geqslant \ln_2(\rho)} \lp{u_{\geqslant k-C}\dot v_k}{H^1_y}
		\ \lesssim \ \sum_{k\geqslant \ln_2(\rho)} \lp{u_{\geqslant k-C}( \dot v_k
		+ D_y^2 \td w_k)}{H^1_y}
		+ \sum_{k\geqslant \ln_2(\rho)} \lp{u_{\geqslant k-C} D_y^2 \td w_k}{H^1_y}
		\ . \label{HL_dS_bound}
\end{equation}
Using paraproduct estimates similar to \eqref{basic_prod_est} and orthogonality
we have the fixed frequency bounds:
\begin{equation}
		\slp{u_{\geqslant k-C}( \dot v_k+ D_y^2 \td w_k)}{H^1_y}  \lesssim 
		\slp{u_{\geqslant k-C} }{H^1_y}
		\slp{  \dot v_k+ D_y^2 \td w_k}{L^\infty_y} \ , \quad
		\slp{u_{\geqslant k-C} D_y^2 \td w_k}{H^1_y}  \lesssim 
		\slp{D_y u_{\geqslant k-C}  }{H^1_y}\lp{  D_y \td w_k}{L^\infty_y} \ . \notag
\end{equation}
Using bounds similar \eqref{L00_sob} we gain a factor of $\rho^\delta 2^{-\frac{1}{2}k}$
from each RHS $L^\infty_y$ term. Therefore:
\begin{equation}
		LHS\eqref{HL_dS_bound} \ \lesssim \ 
		\sum_{k\geqslant \ln_2(\rho)} \rho^{2\delta}2^{-\frac{1}{2}k}\lp{u}{S(\rho)}\lp{v}{\dot S(\rho)}
		\ , \notag
\end{equation}
which suffices to give a uniform $H^1_y$ bound in this case, and in particular one also gains  
control of $L^2_y\cap L^\infty_y$.
\end{proof}

 \subsection{Some Estimates for Coefficients}

We record here a number of asymptotic calculations for  nonlinear
coefficients such as $\beta_1(\rho y)$. 

\begin{lem}[An estimate for the cubic coefficient]
For functions of the form $\varphi(\rho y)$ with $\varphi\in \mathcal{S}$, and 
for $(n,m)\in \mathbb{N}\times \mathbb{Z}$,
we have for $k>0$ the uniform bound: 
\begin{equation}
		\lp{\partial_\rho^n D_y^m P_k \varphi}{L^p_y} \ \leqslant  \ 
		C_{n,m,N}(\lp{ \varphi}{\mathcal{S}})
		\rho^{-n-m-1}  2^{(m+1-\frac{1}{p})k}2^{-N(k-\ln_2(\rho))_+}
		 \ , \label{P_k_beta}
\end{equation}
\end{lem}

\begin{proof}[Proof of lemma]
We have $D_y^m P_k \varphi (\rho y)=(2^k/\rho)^m \td{P}_k {\varphi}(\rho y)$ where 
the symbol of $\td{P}_k$ has  similar support and smoothness properties as $P_k$. 
Differentiation by $\partial_\rho$ 
produces $ \partial_\rho D_y^m P_k \varphi (\rho y)= 2^{mk}\rho^{-m-1}\td{P}_k \td{\varphi}$
where $\td{\varphi}(x)=-m\varphi(x) +x\varphi'(x)$. Thus, by induction one has
$\partial_\rho^n D_y^m \varphi (\rho y)=\rho^{-n} (2^k/\rho)^m \td{P}_k \psi(\rho y)$
for some $\psi\in \mathcal{S}$ and  
$\td{P}_k$ with similar  properties as $P_k$. 
Next, we have from change of variable
and the $L^1_x\to L^p_x$
Bernstein's inequality:
\begin{align}
		\lp{\td P_k \psi(\rho y)}{L^p_y} \ = \ \rho^{-\frac{1}{p}}\lp{\td P_{k-\ln_2(\rho)}\psi}{L^p_x}
		\ &\lesssim \ \rho^{-\frac{1}{p}} 2^{-N(k-\ln_2(\rho))_+}
		\lp{\td P_{k-\ln_2(\rho)}\langle D_x \rangle^N \psi}{L^p_x} \ , \notag\\
		&\leqslant \ C_N(\lp{\psi}{\mathcal{S}})\rho^{-1} 2^{(1-\frac{1}{p})k}2^{-N(k-\ln_2(\rho))_+}
		\ . \notag
\end{align}
\end{proof}

\begin{lem}[A stationary phase calculation]\label{SP_lem}
Let $\lambda>1$ and let $h_\lambda(t)$ be a smooth function supported where $t\approx \lambda$ 
with $|\partial_t^n h_\lambda|\lesssim \lambda^{-n}$. 
Let 
$\varphi \in C_0^\infty$ 
%with the property that $\varphi(\xi)\equiv 0$ when $|\xi|<c$ for some $c>0$
and set:
\begin{equation}
		I_\pm(t,x) \ = \ \int_1^t\!\!\! \int e^{i(\pm(t-s)\langle \xi \rangle + 3s
		+ x\xi )} h_\lambda(s)\varphi(\xi)d\xi ds
		\ . \notag
\end{equation}
Then for fixed $\epsilon>0$  sufficiently small one has the asymptotic estimate:
\begin{equation}
		|D_t^n D_x^m I_\pm (t,x)| \ \leqslant \ C_{n,m,N}(\varphi ) \big(
		\chi_{[-1+\epsilon,1-\epsilon]}( x/t) + \langle(t,x)\rangle^{-N}
		\big)\ .
		\label{main_SP_est}
\end{equation}
Here the estimates are uniform over any collection of functions $h_\lambda$
which satisfy the above derivative bounds with a uniform constant.
\end{lem}

 \begin{proof}[Proof of lemma]
Note first that we may reduce to the case $n=m=0$, because differentiation by 
$D_x$ produces an integral of the same form, while differentiation by $D_t$
produces an integral of the same form or a boundary term that is
$e^{i3t}h_{\lambda}(t)\hat \varphi(-x)$ which obviously satisfies estimate 
\eqref{main_SP_est}.

Next, we dispense with two easy special cases:

\case{1}{The region $|x/t|>1-\epsilon$ for $0< \epsilon\ll 1$}
If $\epsilon\ll 1$ is sufficiently small then for either $I_\pm$ in this region
one has $|\partial_\xi \psi|\approx |x|$ where $\psi$ denotes the phase.
Thus $|I_\pm| = O(\langle x\rangle^{-N})$ follows at once from integration by parts.

\case{2}{The integral $I_-$ where $|x/t|<1$}
In this case the phase is non-stationary so one gets a uniform bound via
integration by parts on time with respect to $s$.

Now we move to the main case:

\case{3}{The integral $I_+$ where $|x/t|<1$}
Writing the phase as $\psi(s,\xi;t,x)$ we have:
\begin{equation}
		\psi_s \ = \ -\langle \xi\rangle +3
		\ , \qquad
		\psi_\xi  \ = \ \langle \xi \rangle^{-1}( \xi (t-s) +\langle \xi \rangle x)
		\ . \notag
\end{equation}
So the stationary points 
are  $\xi_\pm =\pm \sqrt{8}$ and  $s_\pm=t\pm \frac{3}{\sqrt{8}}x$.
A rough form of the Taylor expansion of $\psi$ close to its stationary points is:
\begin{equation}
		\psi(s,\xi;t,x) \ = \ \psi(s_\pm,\xi_\pm; t,x) + (s-s_\pm)(\xi-\xi_\pm)q_1(\xi)
		+ x (\xi-\xi_\pm)^2 q_2(\xi) \ , \label{rough_taylor_pase}
\end{equation}
where $|q_2|\approx |q_1 |\approx 1$ are smooth functions of $\xi$.
To take advantage of this
we break the integral up into two pieces $I_+=I_+^{sp}+ I_+^{rem}$ where
$I_+^{sp}$ denotes the $I_+$ integral with integrand multiplied by a cutoff
 $\chi^{sp}(s,\xi;t,x)$ which is supported on the region:
\begin{equation}
		\mathcal{A}^{sp} \ = \ 
		\{|\xi-\xi_\pm| \ll 1\}\ \cap\  \{ |s-s_\pm| \ll \langle x \rangle \}\ \cap\ 
		\{ |s-1|   \gtrsim \langle x\rangle^\frac{1}{2} \}\cap 
		\{ |t-s| \gtrsim \langle x\rangle^\frac{1}{2} \}  \ , \notag %\label{sp_region}
\end{equation}
and chosen in such a way that $|\partial_\xi^a \partial_s^b \chi^{sp}|
\lesssim \langle x\rangle^{-\frac{1}{2}b}$. The analysis proceeds via
 further sub-cases:

\case{3a.1}{The estimate for $I_+^{rem}$ when $|\xi-\xi_\pm|\gtrsim 1$}
This is similar to the estimate for $I_-$ above, just integrate by parts one time with respect
to $\partial_s$.

\case{3a.2}{The estimate for $I_+^{rem}$ when $|\xi-\xi_\pm|\ll 1$ and $|s-s_\pm|\gtrsim \langle
x \rangle$}
Here the rough Taylor expansion \eqref{rough_taylor_pase} 
shows that the differentiated phase $\psi_\xi$ is $O(s-s_\pm)$.
Therefore,  integrating by parts twice with respect to $\xi$ yields an integrable
weight of the form  $\langle s-s_\pm \rangle^{-2}$.

\case{3a.3}{The estimate for $I_+^{rem}$ when $|\xi-\xi_\pm|\ll 1$ and $|s-s_\pm|\ll \langle
x \rangle$, and either $|s-1|\ll \langle x\rangle^\frac{1}{2}$ or $|t-s|\ll \langle x\rangle^\frac{1}{2}$}
 A standard stationary phase calculation in one 
variable shows:
\begin{equation}
		\big|
		\int e^{i(\pm(t-s)\langle \xi \rangle + 3s
		+ x\xi )} 
		(1-\chi^{sf}(s,\xi;t,x))
		h_\lambda(s)\varphi(\xi)d\xi 
		\big| \ \lesssim \ \langle x\rangle^{-\frac{1}{2}} \ , \notag
\end{equation} 
which is sufficient to integrate over this region.

\case{3b}{The estimate for $I_+^{sp}$}
The form of the phase on line \eqref{rough_taylor_pase} shows that if we make the
change of variables $s'=\langle x\rangle^{-\frac{1}{2}} (s-s_\pm)$ and
$\xi'=\langle x \rangle^\frac{1}{2}(\xi-\xi_\pm)$ then in the new variables
we have that the hessian $\psi''$ is $O(1)$ and non-degenerate at $s'=\xi'=0$,
and the phase may be written in the form:
\begin{equation}
		\psi \ = \ \psi_0 + a s'\xi' + b \xi'^2 + r(s',\xi') \ , \qquad \hbox{where \ \ } 
		a \neq 0 \ , \notag
\end{equation}
and the remainder satisfies $|\nabla^\alpha r(s',\xi')|\ll (|s'|+|\xi'|)^{2-|\alpha|}$
on the region $\mathcal{A}^{sp}$ defined above. Moreover on this region 
in the $(s',\xi')$ variables one has that the amplitude 
$\Phi=\chi^{sp}h_\lambda \varphi$ obeys derivative bounds 
of the form $|\nabla^\alpha_{(s',\xi')} \Phi|\lesssim 1$. Thus, after a further 
uniformly smooth change of variables
we may write:
\begin{equation}
		I_+^{sp} \ = \ e^{i\psi_0}\int\!\!\!\int e^{i(z_1^2-z_2^2)}\Phi(z)dz \ , 
		\qquad \hbox{where \ \ } \Phi\in C_0^\infty \hbox{ \ \and \ \ }
		|\nabla_z^\alpha \Phi| \lesssim 1 \ , \notag
\end{equation}
and the uniform bound $I_+^{sp} = O(1)$ follows from integration by parts away from the origin.

\end{proof}

%-------------------------------------------------------------------------

\subsection{Proof of the Nonlinear $L^2$ and $L^\infty$ Estimates}\label{Igor_sect}

Here we prove Proposition \ref{nonlin_est_prop}. 
First note that  \eqref{nonlin_nrg_est1}  is a standard energy estimate 
and follows  via multiplication of the equation \eqref{nonlin1}
 by $\dot u$ and then integrating by parts. Details are left to the reader.

We obtain the $L^\infty$ bound \eqref{nonlin_nrg_est2} 
in essentially the same way after a frequency localization. Notice that
by estimate \eqref{L00_sob} and the form of RHS \eqref{nonlin_nrg_est2} 
we only need to establish this for frequencies $k<c\ln_2(\rho)$, as long as 
we choose $c$ so that $\delta \ll c \ll 1$. Our goal then is to show:
\begin{equation}
		\slp{(w_{<c\ln_2(\rho)},\dot w_{<c\ln_2(\rho)})}{  L^\infty_{\rho,y} [1,T]} 
		\!\lesssim\!  \slp{(w_0,w_1)}{H^2_y \times H^1_y} 
		\!+ \!\!\!\sup_{1\leqslant \rho\leqslant T}\!\!\!\! \rho^{-\td{c}}\slp{w(\rho)}{S(\rho)}
		\!+\! \mathcal{Q}_2\big(\slp{w}{S[1,T]} \big)\! +\!
		\slp{G}{L^1_\rho (  L^\infty_y)  [1,T]}  , \label{nonlin_nrg_est2'}
\end{equation}
where $\td{c}=\frac{1}{8}c$ and
where our convention  is  $\dot w_{<k}=P_{<k}\partial_\rho  w$.
To this end we write the $P_{<c\ln_2(\rho)}$ truncated version of   \eqref{nonlin2} as:
\begin{equation}
		(\Box_\mathcal{H}+1)w_{<c\ln_2(\rho)} \!+\! \rho^{-1} a_{<c\ln_2(\rho)+C}
		w_{<c\ln_2(\rho)}^3 \!+\! \rho^{-1}b_{<c\ln_2(\rho)+C}
		\dot w_{<c\ln_2(\rho)}^2w_{<c\ln_2(\rho)} \! 
		= \! \mathcal{R}_c (w,\dot w) \!+\! G_{<c\ln_2(\rho)} \ . \label{freq_loc_evol}
\end{equation}

The   proof of \eqref{nonlin_nrg_est2'} is in a series of steps.

\step{1}{Estimate for the remainder}
Here we establish the preliminary  estimate:
\begin{equation}
		\lp{  \mathcal{R}_c (w,\dot w) }{L^\infty_y}
		\ \lesssim \ \rho^{-1-\frac{1}{2}c+\delta}
		\mathcal{Q}_1(\lp{w}{S(\rho)}) \ . \label{L00_nonlin_rem}
\end{equation}
This is shown for the various parts of $\mathcal{R}_c$ separately.

\case{1}{The term $[\partial_\rho^2 , P_{<c\ln_2(\rho)}]w$}
We have $[\partial_\rho^2 , P_{<c\ln_2(\rho)}]w = \rho^{-1} P'_{k=c\ln_2(\rho) + O(1)}\dot w
+ \rho^{-2} P''_{k=c\ln_2(\rho) + O(1)}w$ for cutoffs $P,P''$ with standard $L^1$ kernels.
Then \eqref{L00_nonlin_rem} for this term follows from estimate
\eqref{L00_sob}.

\case{2}{The term $T=\rho^{-1}P_{<c\ln_2(\rho)}\big( a w^3 \big)
- \rho^{-1}  a_{<c\ln_2(\rho)+C} w^3_{<c\ln_2(\rho)} $}
We decompose this further into two subpieces by writing $T=T_{1}-T_{2}$
where:
\begin{equation}
		T_{1} \ = \ 
		\rho^{-1}P_{<c\ln_2(\rho)}\big[ a(w^3- w^3_{<c\ln_2(\rho)})\big] \ , \qquad
		T_{2} \ = \ \rho^{-1}P_{\geqslant c\ln_2(\rho)}\big( 
		a_{<c\ln_2(\rho)+C} w^3_{<c\ln_2(\rho)}\big) 
		\ . \notag
\end{equation}

\case{2.a}{The bound for $T_{1}$}
Here we write  $w^3 -(w_{<k})^3=(w_{\geqslant k})^3
+3(w_{\geqslant k})^2 w_{<k}+3 w_{\geqslant k}w_{<k}^2$,
then \eqref{L00_nonlin_rem} follows from summing over the Sobolev estimate
\eqref{L00_sob} for one of the $w_{\geqslant c\ln_2(\rho)}$.

\case{2.b}{The bound for $T_{2}$}
We may write $ a_{<c\ln_2(\rho)}=P_{k=c\ln_2(\rho)+O(C)}a_1 + P_{<c\ln_2(\rho)-C}a$.
Since we are assuming $a_1\in \mathcal{S}$ the coefficient bound \eqref{P_k_beta}
gives \eqref{L00_nonlin_rem} for the first term on the RHS of this formula. To handle the second term
we write:
\begin{equation}
		P_{k\geqslant c\ln_2(\rho)}(a_{<c\ln_2(\rho)-C}
		w^3_{<c\ln_2(\rho)}) \ = \ 
		P_{k\geqslant c\ln_2(\rho)}\big(a_{<c\ln_2(\rho)-C} (
		w^3_{<c\ln_2(\rho)} -w^3_{<c\ln_2(\rho)-C}
		)\big)
		\ , \notag
\end{equation}
and then we are again in the situation of \textbf{Case 2.a} above.

\case{3}{The term $\rho^{-1}P_{<c\ln_2(\rho)}\big( b \dot w^2 
w \big)-\rho^{-1}  b_{<c\ln_2(\rho)+C} \dot w^2_{<c\ln_2(\rho)}
w_{<c\ln_2(\rho)} $}
This is similar to the previous case.

\step{2}{The integral estimate}
We multiply the equation \eqref{freq_loc_evol} by $\partial_\rho w_{<c\ln_2(\rho)}$
and then first  integrate by parts in time followed by taking the sup in $y$. This yields
the estimate:
\begin{equation}
		\lp{(\partial_\rho w_{<c\ln_2(\rho)},  w_{<c\ln_2(\rho)})}{L^\infty_y}^2\big|_1^T
		\ \lesssim \   I_1  +  I_2 + N_1+N_2   \ , \label{IBP_iden}
\end{equation}
where:
\begin{align}
		I_1 \ &= \ \sup_y \big|   \int_1^T  \rho^{-1} a_{<c\ln_2(\rho)+C} w_{<c\ln_2(\rho)}^3\partial_\rho
		w_{<c\ln_2(\rho)}d\rho\big| \ , \notag \\ 
		I_2 \ &= \ \sup_y \big|   \int_1^T  \rho^{-1}b_{<c\ln_2(\rho)+C}
		\dot w_{<c\ln_2(\rho)}^2w_{<c\ln_2(\rho)}\partial_\rho  w_{<c\ln_2(\rho)} d\rho \big|
		\ , \notag
\end{align}
and:
\begin{equation}
		N_1 \! = \! \lp{\rho^{-2}(w_{<c\ln_2(\rho)},\partial_y^2w_{<c\ln_2(\rho)}   
		)\partial_\rho
		w_{<c\ln_2(\rho)} }{L^1_\rho(L^\infty_y)[1,T]} \ , \quad
		N_2 \! = \! \lp{(\mathcal{R}_c, G_{<c\ln_2(\rho)})\partial_\rho
		w_{<c\ln_2(\rho)} }{L^1_\rho(L^\infty_y)[1,T]} \ . \notag
\end{equation}
Our goal for the remainder of the proof is to show:
\begin{equation}
		 I_1 + I_2  + N_1 + N_2  \lesssim 
		 \epsilon \lp{\partial_\rho w_{<c\ln_2(\rho)}}{L^\infty_{\rho,y}[0,T]}^2
		\!+\! \sup_{0\leqslant \rho\leqslant T} \rho^{-2\td{c}}\lp{w(\rho)}{S(\rho)}^2 
		+ \mathcal{Q}_4(\lp{w}{S[1,T]})
		+ \epsilon^{-1} \lp{G}{L^1_\rho(L^\infty_y)}^2 
		\ , \label{nonlin_L00_error_est}
\end{equation}
which will be done for each term separately. With \eqref{nonlin_L00_error_est} in hand, estimate
\eqref{nonlin_nrg_est2'} follows from supping \eqref{IBP_iden} over times in the interval $[0,T]$,
taking $\epsilon$ sufficiently small to absorb the first term on RHS \eqref{nonlin_L00_error_est},
and then using  commutator formula:
\begin{equation}
		\dot w_{<c\ln_2(\rho)}\ = \ 
		\partial_\rho w_{<c\ln_2(\rho)}-\rho^{-1}P'_{k=c\ln_2(\rho)+O(1)}w 
		\ , \label{dot_w_exp}
\end{equation}
to return to the original $ \dot w_{<c\ln_2(\rho)}$ variable.

\case{1}{The integral $I_1$}
Here we show $I_1\lesssim \mathcal{Q}_4(\lp{w}{S[1,T]})$. To prove it 
we integrate by parts one time with respect to $\rho$, noting that the weight
$|\partial_\rho (\rho^{-1} a_{<c\ln_2(\rho)+C})|\lesssim \rho^{-2}$ is integrable.

\case{2}{The integral $I_2$}
This term is a little more involved. First use the identity (for time dependent frequency cutoff
and $b_{<c\ln_2(\rho)+C}$ replaced by $b$ to save notation):
\begin{equation}
		\rho^{-1}b \dot w_{<k}^2 w_{<k}\partial_\rho w_{<k}=
		\frac{1}{2}\partial_\rho \big(\rho^{-1}b \dot w_{<k}^2 w_{<k}^2\big) 
		- \frac{1}{2}
		\partial_\rho (\rho^{-1}b)  \dot w_{<k}^2 w_{<k}^2
		- \rho^{-1}b\dot w_{<k}\partial_\rho \dot w_{<k} w_{<k}^2 \ . \notag
\end{equation}
The integral of the first and second  terms on the RHS above are directly bounded 
in terms of RHS \eqref{nonlin_L00_error_est}. For the last term on RHS above
we use the identity:
\begin{equation}
		\partial_\rho \dot w_{<c\ln_2(\rho)}=-(1+4^{-1}\rho^{-2})w_{<c\ln_2(\rho)}
		+ (\Box_\mathcal{H} + 1)w_{<c\ln_2(\rho)}
		-\partial_\rho [\partial_\rho, P_{<c\ln_2(\rho)}]w 
		+ \rho^{-2}\partial_y^2 w_{<c\ln_2(\rho)} \ . \label{ddot_w_exp}
\end{equation}
All of the terms on the RHS \eqref{ddot_w_exp} yield integrable contributions
when multiplied by $\rho^{-1}b \dot w_{<c\ln_2(\rho)} w_{<c\ln_2(\rho)}^2$, possibly
after an addition integration by parts. We explain each case separately:

For the first term on RHS \eqref{ddot_w_exp} multiplied by $\rho^{-1}b \dot 
w_{<c\ln_2(\rho)} w_{<c\ln_2(\rho)}^2$ we either have an integrable contribution thanks 
to the weight $\rho^{-2}$, or we are back in the situation of \textbf{Case 1} above after
a use of identity \eqref{dot_w_exp} to write things as an absolutely integrable contribution
plus a time derivative.

For the second term on RHS \eqref{ddot_w_exp} multiplied by $\rho^{-1}b \dot 
w_{<c\ln_2(\rho)} w_{<c\ln_2(\rho)}^2$ we have an integrable contribution by 
inspection of the terms in equation \eqref{freq_loc_evol} and using 
estimate \eqref{L00_nonlin_rem}. 

For the third
term on RHS \eqref{ddot_w_exp} multiplied by $\rho^{-1}b \dot 
w_{<c\ln_2(\rho)} w_{<c\ln_2(\rho)}^2$ we expand:
\begin{equation}
		\partial_\rho [\partial_\rho, P_{<c\ln_2(\rho)}] \ = \ 
		\rho^{-1} P'_{k=c\ln_2(\rho) + O(1)} (\partial_\rho - \rho^{-1}) + \rho^{-2}
		P''_{k=c\ln_2(\rho) + O(1)} \ , \notag
\end{equation}
which directly leads to absolutely integrable contributions.

Finally, for the last  term on RHS \eqref{ddot_w_exp} multiplied by $\rho^{-1}b \dot 
w_{<c\ln_2(\rho)} w_{<c\ln_2(\rho)}^2$ we first sum over 
 \eqref{L00_sob} to get:
\begin{equation}
		\lp{\rho^{-2}\partial_y^2 P_{<c\ln_2(\rho)} w}{L^\infty} \ \lesssim \ \rho^{-2+\frac{3}{2}c+\delta}
		\lp{w}{S(\rho)} \ . \label{low_dy_Sob} 
\end{equation}
Again this leads to an absolutely integrable contribution.

Note that up to this point all of our estimate have been in terms of the last two
expressions on RHS \eqref{nonlin_L00_error_est}.

\case{3}{The term $N_1$}
For this we use identity \eqref{dot_w_exp} and estimate \eqref{low_dy_Sob}
which directly gives the bound  
$N_1\lesssim \sup_{0\leqslant \rho\leqslant T} \rho^{-c}\lp{w(\rho)}{S(\rho)}^2 $.

\case{4}{The term $N_2$}
For this we directly use estimate \eqref{L00_nonlin_rem} and Young's inequality which gives:
\begin{equation}
		N_2 \ \lesssim \ \epsilon \lp{\partial_\rho w_{<c\ln_2(\rho)}}{L^\infty_{\rho,y}[0,T]}^2
		+ \sup_{0\leqslant \rho\leqslant T} \rho^{-2\td{c}}\lp{w(\rho)}{S(\rho)}^2 
		+\mathcal{Q}_4(\lp{w}{S[1,T]})
		+ \epsilon^{-1} \lp{G}{L^1_\rho(L^\infty_y)}^2 \ . \notag
\end{equation}

%-------------------------------------------------------------------------

\subsection{Multilinear Estimates Pt. II: 
Semi-classical Bilinear Operators}

To set this up we define the $h=\rho^{-1}$ semi-classical fourier transform
as:
\begin{equation}
    	\widehat{u}(\xi) \ = \ \mathcal{F}_\rho (v)(\xi) \ = \ 
	 \rho\, \int_\mathbb{R}\  e^{-i\rho \xi y}\ u(y) \ dy \ , \notag %\label{SC_F}
\end{equation}
with inversion formula:
\begin{equation}
     u(y) \ = \ \mathcal{F}^{-1}_\rho (\widehat{u})(y) \ = \
    \frac{1}{2\pi}\ \int_\mathbb{R}\  e^{i\rho y  \xi }\ \widehat{u}(\xi) \ d\xi \ . \notag %\label{SC_Finv}
\end{equation}
In particular the symbol of $D_y$ is $\xi$ in this picture. One must keep 
in mind that the operation $u\to \widehat{u}$ is time dependent, in particular one  has 
the identity:
\begin{equation}
		 [\partial_\rho , \mathcal{F}_\rho]  \ = \ \rho^{-1} \partial_\xi \xi  
		 \mathcal{F}_\rho \ . \label{FT_comm}
\end{equation}
% $ -i \mathcal{F}_\rho\, yD_y$.
Next, given a symbol $k(\xi,\eta)$ we use it to define a semi-classical bilinear operator
as follows:
\begin{equation}
    K( {}^{1\!}D_y,{}^{2\!}D_y)[u,v] \ = \ \frac{1}{4\pi^2}\,
    \int\!\!\! \int k(\xi,\eta)
    e^{i\rho y(\xi+\eta)}
    \widehat{u}(\xi) \widehat{v}(\eta)\ d\xi d\eta \ . \label{pdo_def}
\end{equation}
The class of symbols we will work with are the following. 

\begin{defn}[$S^{a,b;c}$ symbols]
We say that a symbol $k(\xi,\eta)$ is in $S^{a,b;c}$ if $k=p/q$ where
$q \approx  (\langle \xi \rangle + \langle\eta \rangle )^c$ 
and:
\begin{equation}
		|(\langle \xi\rangle \partial_\xi)^l (\langle \eta\rangle \partial_\eta)^m p | \ \lesssim \ C_{l,m}
		\langle \xi \rangle ^a \langle \eta \rangle ^b \ , 
		\qquad |(\langle \xi\rangle \partial_\xi)^l (\langle \eta\rangle \partial_\eta)^m q | \ \lesssim \ C_{l,m}
		(\langle \xi \rangle + \langle\eta \rangle )^c \ . \label{bilin_sym_semi}
\end{equation}
We say that $k\in S_0^{a,b;c}$ if in addition to the above bounds one has $k(0,0)=0$.
\end{defn}

Our main algebraic result for the operators $K({}^1\!D_y,{}^2\!D_y)$
is the following.

\begin{lem}[Bilinear $\Psi$DO Calculus]
Let $k\in S^{a,b;c}$ and let  $K({}^1\!D_y,{}^2\!D_y)$
be the corresponding operator operator as defined on line \eqref{pdo_def}. 
Then for all $u,v\in\mathcal{S}$ one has the following
identities:
\begin{equation}
        D_y K( {}^{1\!}D_y,{}^{2\!}D_y)[u,v] \ = \
     K( {}^{1\!}D_y,{}^{2\!}D_y)[D_y u,v]
    + K( {}^{1\!}D_y,{}^{2\!}D_y)[u,D_y v] \ . \label{leib1}
\end{equation}
and:
\begin{multline}
        \partial_\rho K( {}^{1\!}D_y,{}^{2\!}D_y)[u,v] \ = \
     K( {}^{1\!}D_y,{}^{2\!}D_y)[\partial_\rho u,v]
    + K({}^{1\!}D_y,{}^{2\!}D_y)[u,\partial_\rho v]\\
    - \rho^{-1}\partial_1 K( {}^{1\!}D_y,{}^{2\!}D_y)[D_y u,v]
    - \rho^{-1}\partial_2 K( {}^{1\!}D_y,{}^{2\!}D_y)[u,D_y v]
    \ . \label{leib2}
\end{multline}
Here  $\partial_1 K$ and $\partial_2 K$
have symbols $\partial_\xi k$ and $\partial_\eta k$ respectively.
\end{lem}

\begin{proof}
The first identity follows at once by differentiation of the formula \eqref{pdo_def}.
The second formula follows as well by combining \eqref{pdo_def} with \eqref{FT_comm},
and then integrating by parts with respect to $\xi$.
\end{proof}

Our main analytical result is:

\begin{prop}[Estimates for Bilinear $\Psi$DO]\label{main_bilin_est_prop}
Let $a,b,c\geqslant 0$ and let $k\in S^{a,b;c}$ with $K=K({}^1\!D_y,{}^2\!D_y)$
the corresponding operator defined by $\mathcal{F}_\rho$. Then 
one has the following mappings:
\begin{align}
		K:\,  S(\rho) \times S(\rho) \ &\longrightarrow \ S(\rho) \ , 
		&\hbox{if \ \ \ }& a+b\ \leqslant 1+c \ , \quad
		a \ \leqslant \ c \ , \qquad \ \ 
		b \ \leqslant \ c \ , \label{bilin_est1}\\
		K:\,  S(\rho) \times \dot{S}(\rho) \ &\longrightarrow \ S(\rho) \ , 
		&\hbox{if \ \ \ }& a+b\ \leqslant c \ , \qquad\ \ 
		a \ \leqslant \ c \ , \qquad \ \ 
		b \ \leqslant \ c-1 \ , \label{bilin_est2}\\
		K:\,  \dot S(\rho) \times \dot S(\rho) \ &\longrightarrow \ S(\rho) \ , 
		&\hbox{if \ \ \ }& a+b\ \leqslant c-1 \ , \quad
		a \ \leqslant \ c-1 \ , \quad 
		b \ \leqslant \ c-1 \ . \label{bilin_est3}
\end{align}
Finally, if $k\in S_0^{a,b;c}$ then under the same restrictions as
above one has the improved $L^\infty$ bound:
\begin{equation}
		K:\,  A \times B \ \longrightarrow \ \rho^{-\frac{1}{2}+\delta}L^\infty_y \ , 
		\qquad A=S(\rho) ,\dot S(\rho) \ , \ \ B=S(\rho) ,\dot S(\rho) \ . \label{improved_L00}
\end{equation}
\end{prop}

The last estimate of this proposition has a slightly more useful form that we will
use in the sequel:

\begin{prop}\label{imp_L00_prop}
Let $k\in S^{a,b;c}$ then under the same restrictions as on lines 
\eqref{bilin_est1}--\eqref{bilin_est3}. Then for $A,B$
any combination of $A=S(\rho) ,\dot S(\rho)$ and $B=S(\rho) ,\dot S(\rho)$
we have:
\begin{equation}
		\lp{K({}^1\!D_y,{}^2\!D_y)[u,v]-k(0,0)uv}{L^\infty_y} \ \lesssim \ 
		\rho^{-\frac{1}{2}+\delta}\lp{u}{A}\lp{v}{B} \ . \label{imp_L00_est}
\end{equation}
\end{prop}

To prove  Proposition \ref{main_bilin_est_prop} we will rely on two lemmas.

\begin{lem}[Dyadic Kernel Bounds]\label{dyadic_K_lem}
Let $k\in S^{a,b;c}$ with $a,b,c\geqslant 0$, and let
$K$ be the corresponding operator in the case $\rho=1$.
Then for all $1\leqslant p,q,r\leqslant \infty$ with $p^{-1}+q^{-1}=r^{-1}$, and
integers $0\leqslant k\leqslant k'$, one has the uniform family of bounds:
\begin{equation}
		K:\, P_k L^p\times P_{k'}L^q \ \longrightarrow \ 
		2^{a k}2^{(b-c)k'} L^r \ . \label{dyadic_K_bound}
\end{equation}
Here we set $P_{k}=P_{\leqslant 0}$ in the case $k=0$ and similarly for $k'$.
\end{lem}

\begin{lem}[Model Estimates]\label{model_est_lem}
Let $k\in S^{a,b;c}$ with $a,b,c\geqslant 0$ and such that $a+b\leqslant c+1$, $a\leqslant c$, 
and $b\leqslant c$. Let $K$ be the corresponding operator in the case $\rho=1$. Then one
has the three embeddings:
\begin{align}
		K:\, H^1 \times \dot{H}^1\cap L^\infty \ &\longrightarrow \ L^2 \ , \label{model_est1}\\
		K:\, L^2 \times \dot{H}^1\cap \dot{H}^2 \cap L^\infty \ 
		&\longrightarrow \ L^2 \ , \label{model_est2}\\
		K:\,   \dot{H}^1\cap L^\infty\times\dot{H}^1\cap \dot{H}^2 \cap L^\infty  
		\ &\longrightarrow \ L^\infty \ . \label{model_est3}
\end{align} 
\end{lem}

\begin{proof}[Proof of Lemma \ref{dyadic_K_lem}]
The symbol of $K({}^1\! D,{}^2\! D)[P_k u, P_{k'} v]$ is $l(\xi,\eta)=k(\xi,\eta)p_k(\xi)p_{k'}(\eta)$, 
and one directly verifies that $2^{-ak+(c-b)k'}l \in S^{0,0;0}$, so it suffices 
to prove estimate \eqref{dyadic_K_bound} in this case. 

Assuming now that $k\in S^{0,0;0}$, we have that the convolution kernel of 
the localized operator  $L({}^1\! D,{}^2\! D)$ is given by:
 \begin{equation}
 		L({}^1\! D ,{}^2\! D)(x,y) \ = \ \frac{1}{4\pi^2} \int\!\!\!\!\int 
		e^{i(x\xi+y\eta)}k(\xi,\eta)p_k(\xi)p_{k'}(\eta)
		d\xi d\eta \ , \notag
 \end{equation}
 so it suffices to show $k\in S^{0,0;0}$ implies $\lp{L}{L^1(dxdy)}\lesssim 1$ with a unifrom
 bound depending on only finitely many of the seminorms \eqref{bilin_sym_semi}. Integrating by
parts two times when either $|x|\geqslant 2^{-k}$ or $|y|\geqslant 2^{-k'}$, and doing nothing otherwise,
and evaluating the resulting integral gives:
\begin{equation}
		 | L({}^1\! D ,{}^2\! D)(x,y) | \ \lesssim \ 2^{k+k'}(1+ 2^{k} |x|)^{-2}(1+ 2^{k'} |y|)^{-2}
		 \ . \notag
\end{equation}
The desired $L^1$ bound follows.
\end{proof}

\begin{proof}[Proof of Lemma \ref{model_est_lem}]
The proof contains a number of steps and subcases. Here we use the convention
that $P_0=P_{\leqslant 0}$ in all Littlewood-Paley decompositions.

\step{1}{Low frequency factor}
We first dispense with the special case of the contribution of $P_{\leqslant 0}v$.
Breaking $u$ up into dyadic frequencies it suffices show:
\begin{equation}
		\lp{\sum_{k\geqslant 0} 
		K({}^1\! D,{}^2\! D)[P_k u, P_{\leqslant 0} v] }{L^2}  \lesssim  
		\lp{u}{L^2}\lp{v}{L^\infty} \ , \quad
		\lp{\sum_{k\geqslant 0} 
		K({}^1\! D,{}^2\! D)[P_k u, P_{\leqslant 0} v] }{L^\infty}  \lesssim  
		\lp{u}{L^\infty\cap \dot H^1}\lp{v}{L^\infty} \ . \notag
\end{equation}
The first estimate follows from \eqref{dyadic_K_bound} and the fact that
the LHS sum is essentially orthogonal. The second bound follows by breaking the
sum into $k<C$ and $k>C$, and using \eqref{dyadic_K_bound}
in the $L^\infty\cdot L^\infty$ form in both cases followed by:
\begin{equation}
		\sum_{k>C}\lp{P_k u}{L^\infty} \ \lesssim \ \
		\sum_{k>C}2^{-\frac{1}{2}k }\lp{P_k u}{\dot H^1} \ \lesssim \ \lp{u}{\dot H^1} \ , \notag
\end{equation}
in the case of $P_{>C}u$.

It now suffices to prove \eqref{model_est1} and \eqref{model_est2} when the second
factor is $P_{>0}v$, and \eqref{model_est3} in the case where the factors are 
$P_{>0}u$ and $P_{>0}v$ (by symmetry).

\step{2}{Estimates  \eqref{model_est1} and \eqref{model_est2} with $P_{>0}v$}
Because of the frequency restriction on $v$ we may assume its norm
is $H^1$  in $\eqref{model_est1}$, and $H^2$ in  $\eqref{model_est2}$.
Now break the product up into a sum over all frequencies and group terms:
\begin{equation}
		T_1  = \!\!\!\!\! \sum_{\substack{k,k'\geqslant 0\\ k>k'+C}} 
		K({}^1\! D,{}^2\! D)[P_k u, P_{k'} v] \ , \ \ 
		T_2 = \!\!\!\!\! \sum_{\substack{k,k'\geqslant 0\\ k'>k+C}} 
		K({}^1\! D,{}^2\! D)[P_k u, P_{k'} v] \ , \ \ 
		T_3 = \!\!\!\!\! \sum_{\substack{k,k'\geqslant 0\\ |k-k'|\leqslant C}} 
		K({}^1\! D,{}^2\! D)[P_k u, P_{k'} v] 
		\ . \notag
\end{equation}
Now there are two subcases:

\case{1}{High-Low and Low-High products}
Here we consider the sums $T_1$ and $T_2$. Since the sum is essentially orthogonal
in the high frequency factor we can reduce to a fixed high frequency. 
Putting the low frequency factor in $L^\infty$ via Sobolev and using
\eqref{dyadic_K_bound} we have:
\begin{equation}
		\lp{K({}^1\! D,{}^2\! D)[P_k u, P_{k'} v] }{L^2} \ \lesssim \
		2^{\frac{3}{2}  \min\{k,k'\}} \lp{P_k u}{L^2}\lp{P_{k'}v}{L^2}
		\ . \label{basic_L2_product}
\end{equation}
In the case of estimates  \eqref{model_est1} and \eqref{model_est2} we bound 
the RHS above as (resp):
\begin{equation}
		\hbox{RHS}\eqref{basic_L2_product} \ \lesssim \
		2^{-\frac{1}{2}\min\{k,k'\}}2^{-|k-k'|}\lp{P_k u}{H^1}\lp{P_{k'}v}{H^1} \ , \qquad
		\hbox{RHS}\eqref{basic_L2_product} \ \lesssim \
		2^{-\frac{1}{2}\min\{k,k'\}}\lp{P_k u}{L^2}\lp{P_{k'} v}{H^2} \ . \notag
\end{equation}
In either case one may sum the low frequency term.

\case{2}{High-High products}
In this case we bring the sum outside the $L^2$ norm and then use 
\eqref{basic_L2_product} which directly gives $\sum_{k}2^{-\frac{1}{2}k}(RHS\ norms)$
for both   \eqref{model_est1} and \eqref{model_est2}.

\step{3}{Estimate \eqref{model_est3} with
$P_{>0}u$ and $P_{>0}v$}
This is similar to the argument above. First break into all possibly frequency combinations
and bring the resulting sum outside the norm. For fixed frequency 
we get from \eqref{dyadic_K_bound} and Sobolev embedding the estimate:
\begin{equation}
		\lp{K({}^1\! D,{}^2\! D)[P_k u, P_{k'} v] }{L^\infty} \ \lesssim \
		2^{   \min\{k,k'\}} 2^{-\frac{1}{2}k}2^{-\frac{3}{2}k'}\lp{P_k u}{H^1}\lp{P_{k'}v}{H^2}
		\ \lesssim \ 2^{-\frac{1}{2}|k-k'|}\lp{P_k u}{H^1}\lp{P_{k'}v}{H^2}
		\ . \notag
\end{equation}
One may sum the RHS over all $k,k'>0$ using the $\ell^2$ Young's inequality
and orthogonality.
\end{proof}

\begin{proof}[Proof of Proposition \ref{main_bilin_est_prop}]
The proof is in a series of steps.

\step{1}{Reduction and rescaling}
First we note that by the first estimate on line \eqref{dot_S_bnds} we only have to consider 
the $S(\rho)\times S(\rho)$ case of estimates \eqref{bilin_est1}--\eqref{improved_L00}. 

For this estimate we  consider $\dot u$ and $\dot v$ as separate variables, which is permissible 
because \eqref{leib2} shows that differentiation by $\partial_\rho$
preserves the symbol class of $k(\xi,\eta)$.
Next, we rescale everything to $\rho=1$. Setting new variables $x=\rho y$ we have that
$D_y=D_x$ where $D_x=\frac{1}{i}\partial_x$, and the bilinear operators $K({}^1\! D ,{}^2\! D)$
are still defined by \eqref{pdo_def} but now with $\rho=1$. The norm $S(\rho)$ changes under
rescaling to a new norm $\td{S}(\rho)$:
\begin{equation}
		\lp{u}{\td{S}(\rho)} \ = \   \rho^{-\frac{1}{2}}
		\lp{(u,\dot u)}{L^2_x} + \rho^{\frac{1}{2}-\delta} \lp{(u,\dot u, D_x u)}{\dot H^1_x}
		+ \lp{(u,\dot u)}{L^\infty_x } \ . \notag
\end{equation}
In the rescaled picture we are trying to show the two embeddings:
\begin{equation}
		K:\, \td{S}(\rho) \times \td{S}(\rho)
		\ \longrightarrow \ \td{S}(\rho) \ , \qquad \qquad
		K:\,   \td{S}(\rho) \times \td{S}(\rho) \ \longrightarrow \ \rho^{-\frac{1}{2}+\delta}L^\infty \ ,
		\label{rescaled}
\end{equation}
in the case $k\in S^{a,b;c}$ with $a+b\leqslant c+1$, $a\leqslant c$, and $b\leqslant c$ where
in addition we assume  $k\in S^{a,b;c}_0$ for the second estimate. 
In the case of the first bound above we must also consider the case where one of the factors is
$\dot u$ or $\dot v$.
We prove these bounds
separately for each constituent of the $\td{S}(\rho)$ norm.

\case{1}{Contribution of the $L^2$ part}
To estimate the $L^2$ portion of the $\td{S}(\rho)$ norm in the first bound 
on line \eqref{rescaled}  we use the estimates:
\begin{equation}
		\rho^{-\frac{1}{2}} \lp{(u,\dot u)}{H^1} \ \lesssim \
		\lp{u}{\td{S}(\rho)} \ , \qquad
		\lp{v }{L^\infty \cap \dot H^1 }
		\ \lesssim \
		\lp{v}{\td{S}(\rho)} \ , \notag
\end{equation}
where the first bound is used for the time differentiated factor, and
we conclude via estimate 
\eqref{model_est1}.

\case{2}{Contribution of the $\dot H^1$ part} For the $\dot {H}^1$ portion 
of the $\td{S}(\rho)$ norm in the first bound 
on line \eqref{rescaled} we need to differentiate with respect to each of 
$D_x, D_x^2 $. In the case of a time differentiated factor
we only need to apply $D_x$.
There are two cases depending if the derivatives
split under Leibniz's rule, or all go to one factor. In the first case we use:
\begin{equation}
		\lp{(u,\dot u, D_x u )}{L^\infty \cap \dot H^1}
		\ \lesssim \
		\lp{u}{\td{S}(\rho)} \ , \qquad
		\rho^{-\delta + \frac{1}{2}} \lp{D_x u}{H^1} \ \lesssim \
		\lp{u}{\td{S}(\rho)} \ , \notag
\end{equation}
and then conclude via \eqref{model_est1}.
In the second case we use:
\begin{equation}
		\lp{u}{L^\infty \cap \dot H^1\cap \dot H^2 }
		\ \lesssim \ \lp{u}{\td{S}(\rho)} \ , \qquad
		\rho^{-\delta + \frac{1}{2}} 
		\lp{(D_x u , D_x^2 u, D_x\dot{u})}{L^2}
		\ \lesssim \ \lp{u}{\td{S}(\rho)}  \ , \notag
\end{equation}
and then conclude via \eqref{model_est2}.

\case{3}{Contribution of the $L^\infty$ part} To estimate the $L^\infty$ portion of 
the $\td{S}(\rho)$ norm in the first bound 
on line \eqref{rescaled} we  
use:
\begin{equation}
		\lp{u}{L^\infty \cap \dot H^1\cap \dot{H}^2 }
		\ \lesssim \ \lp{u}{\td{S}(\rho)} \ , \qquad
		\lp{\dot u}{L^\infty \cap \dot H^1}
		\ \lesssim \ \lp{u}{\td{S}(\rho)} \ , \notag
\end{equation}
and then conclude via \eqref{model_est3}.

\step{2}{Proof of the improved $L^\infty$ norm}
To prove the second bound on line \eqref{rescaled} it suffices to consider 
instead $P_{\leqslant 0} K({}^1\! D ,{}^2\! D)[u,v]$ because we already have 
the $\td{S}(\rho)$ estimate and can use $P_{>0}\dot H^1\subseteq L^\infty$
%the easy estimate:
%\begin{equation}
%		\lp{P_{>0} K({}^1\! D ,{}^2\! D)[u,v]}{L^\infty} \ 
%		\lesssim \ \rho^{-\delta + \frac{1}{2}}\lp{K({}^1\! D ,{}^2\! D)[u,v]}{\td{S}(\rho)} \ , \notag
%\end{equation}
for the complement. Breaking up the expression into frequencies we have:
\begin{equation}
		P_{\leqslant 0} K({}^1\! D ,{}^2\! D)[u,v] \ = \ P_{\leqslant 0} K({}^1\! D ,{}^2\! D)[P_{<C}u,P_{<C}v]
		+ \!\!\!\sum_{\substack{\max\{ k, k'\} \geqslant C \\ k =k'+O(1)}}
		P_{\leqslant 0} K({}^1\! D ,{}^2\! D)[P_{k}u,P_{k'}v] \ . \label{spec_L00_sum}
\end{equation}
For the second term on RHS we use embedding \eqref{model_est3} and the fact that for high
frequencies one has:
\begin{equation}
		\rho^{-\delta + \frac{1}{2}}\lp{P_k u}{L^\infty \cap \dot H^1\cap \dot{H}^2 }
		\ \lesssim \ \lp{u}{\td{S}(\rho)} \ , 
		\qquad \hbox{for \ \ } k\geqslant C
		\ . \notag
\end{equation}
The frequency sum is then bounded by $\ell^2$ Cauchy-Schwartz and orthogonality.

For the low frequency part on RHS \eqref{spec_L00_sum} 
we need to use the condition $k\in S_0^{a,b;c}$. Assume without
loss of generality that the vanishing condition is due to the first factor, in which case 
the symbol of the total operator is
$p_0(\xi+\eta) k(\xi , \eta)p_{<C}(\xi)p_{<C}(\eta)
= \xi \td{k}(\xi,\eta) p_{<C}(\xi) p_{<C}(\eta)$  
for some smooth compactly supported $\td{k}$. Thus:
\begin{equation}
		\lp{ P_{\leqslant 0} K({}^1\! D ,{}^2\! D)[P_{<C}u,P_{<C}v]}{L^\infty} \ \lesssim \ 
		\lp{D_x P_{<C} u}{L^\infty}\lp{v}{L^\infty} \ , \notag
\end{equation}
and we conclude with $D_x P_{<C}\dot H^1 \subseteq L^\infty$
and $\rho^{\frac{1}{2}-\delta}\td {S}(\rho) \subseteq \dot H^1$.
\end{proof}

Finally, we prove Proposition \ref{imp_L00_prop}.

\begin{proof}
For this proof we switch  back to the $y$ variable.
Given any symbol $k\in S^{a,b;c}$ the new symbol 
$\td{k}(\xi,\eta)=k(\xi,\eta)-k(0,0)p_{<0}(\xi)p_{<0}(\eta)$ is such that $\td {k}\in S_0^{a,b;c}$
with the same indices. Therefore, by estimate \eqref{improved_L00} it suffices
to show that:
\begin{equation}
		\lp{uv - P_{<\ln_2(\rho)}u P_{<\ln_2(\rho)}v }{L^\infty_y} \ \lesssim \ \rho^{-\frac{1}{2}+\delta}
		\lp{u}{A}\lp{v}{B}
		\ , \notag
\end{equation}
for any combination of the cases $A=S(\rho) ,\dot S(\rho)$ and $B=S(\rho) ,\dot S(\rho)$.
Expanding the difference it suffices to show (by symmetry):
\begin{equation}
		\lp{P_{\geqslant \ln_2(\rho)} u }{L^\infty} \ \lesssim \ \rho^{-\frac{1}{2}+\delta} \lp{u}{A} \ , 
		\qquad A=S(\rho) ,\dot S(\rho) \ . \notag
\end{equation}
This follows at once from \eqref{L00_sob}.
\end{proof}

%-------------------------------------------------------------------------
%%%%%%%%%%%%%%%%%%%%%%%%%%%%%% 
%-------------------------------------------------------------------------

\section{The Quadratic Correction}\label{Quad_sect}

The main result of this section is the following:

\begin{thm}[Quadratic Normal Forms]\label{quad_nf_thm}
Let $u$ be a sufficiently smooth and well localized solution to:
\begin{equation}
		(\Box_\mathcal{H} + 1)u \ = \  \rho^{-\frac{1}{2}}\alpha_0 u^2 + \rho^{-1}\beta u^3 + F
		\ , \label{u_quad_nf_def}
\end{equation}
on the time interval $[1,T]$. Then there exists   three nonlinear quantities
$\mathcal{N}_{quad}=\mathcal{N}_{quad}(u,\dot u)$,   
$\mathcal{R}_{quad}=\mathcal{R}_{quad}(u,\dot u,F,\dot{F})$, and 
$\td{\mathcal{R}}_{quad}=\td{\mathcal{R}}_{quad}(u,\dot u,F,\dot{F})$
such that one has the algebraic identity:
\begin{equation}
		(\Box_\mathcal{H} + 1)  \mathcal{N}_{quad} 
		\ = \ \rho^{-\frac{1}{2}}\alpha_0u^2
		  + \mathcal{R}_{quad}
		\ = \ \rho^{-\frac{1}{2}}\alpha_0u^2
		-2\rho^{-1}\alpha_0^2 u^3 + \frac{8}{3}\rho^{-1}
		\alpha_0^2\dot{u}^2 u + \td{\mathcal{R}}_{quad} \ , \label{main_nf_iden_quad'}
\end{equation}
In addition   one has the estimates: 
\begin{multline}
		\lp{\rho^\frac{1}{2} \mathcal{N}_{quad} }{S[1,T]} + \lp{(\mathcal{R}_{quad},
		\dot{\mathcal{R}}_{quad})}{N[1,T]}
		+ \lp{  \td{\mathcal{R}}_{quad}
		 }{ L^1_\rho( L^\infty_y )[1,T]}
		 \ \lesssim \ \lp{\rho^\frac{1}{2} F}{S[1,T]}^2 +
		  \lp{\rho^\frac{1}{2} \dot{F}}{\dot S[1,T]}^2\\
		 +\lp{\rho( F,\dot F)}{L^\infty_{\rho,y}[1,T]}^2
		+ \mathcal{Q}_2\big(\lp{u}{S[1,T]}, \lp{\rho^\frac{1}{2} F}{S[1,T]} ,
		  \lp{\rho^\frac{1}{2} \dot{F}}{\dot S[1,T]}, \lp{\rho( F,\dot F)}{L^\infty_{\rho,y}[1,T]} \big) 
		 \ . \label{main_nf_ests_quad'}
\end{multline}
\end{thm}

We construct $\mathcal{N}_{quad}$ according to the classical method 
of Shatah \cite{Sh}:
\begin{equation}
		\mathcal{N}_{quad} \ = \ \rho^{-\frac{1}{2}}K_0[u,u] + \rho^{-\frac{1}{2}}K_2[\dot u, \dot u]
		\ , \label{N_quad_defn}
\end{equation}
where $K_0,K_2$ are bilinear operators as defined on line \eqref{pdo_def},
which we also assume to be  symmetric.
 To compute them we use
the notation $[\partial_\rho^2,\rho^{-\frac{1}{2}}K]$ to denote the operator:
\begin{equation}
		\big[\partial_\rho^2,\rho^{-\frac{1}{2}}K\big] [u,v] \ =\ \partial_\rho^2 \big(
		\rho^{-\frac{1}{2}} K[u,v] \big) 
		-  \rho^{-\frac{1}{2}} K[\ddot u, v]
		-  \rho^{-\frac{1}{2}} K[u,\ddot v] 
		- 2\rho^{-\frac{1}{2}} K[\dot u, \dot v] \ . \notag
\end{equation}
Then a short calculation shows that for $\mathcal{N}_{quad}$ as defined above:
\begin{equation}
		(\Box_\mathcal{H}+1)\mathcal{N}_{quad} \ = \ T_1 + T_2 + T_3  - \frac{1}{4}\rho^{-2}
		\mathcal{N}_{quad} \ , \notag
\end{equation}
where if we set $G=(\Box_\mathcal{H}+1)u$ then:
\begin{align}
		T_1 \! &= \! \rho^{-\frac{1}{2}}\Big[  2 K_2[ (D_y^2\! +\! 1)u,(D_y^2\!+\! 1)u]
    		\! +\! 2 K_0[D_y u, D_y u] \! +\!
    		2 K_2[D_y \dot{u}, D_y \dot{u}] 
		\!-\! K_0[u,u] \!+\!
    		2K_0[\dot{u},\dot{u}] \! -\!  K_2[\dot{u},\dot{u}]
		\Big] \ , \notag\\
		T_2 \! &= \!  \rho^{-\frac{1}{2}}
		\Big[ 2K_0[G,u]   
    		\! +\! 2 K_2[G,G] \! +\! 2K_2[\dot{G},\dot{u}]  
    		\!-\! 4 K_2[(D_y^2 \!+\! 1)u,G\big) \!- \!\rho^{-2}K_2[G,u]\Big] 
		\ , \notag\\
		T_3 \! &= \! 
		\big[\partial_\rho^2,\rho^{-\frac{1}{2}}\!K_0 \big]\! [u,u ] \!\! +\!\!
		\big[\partial_\rho^2,\rho^{-\frac{1}{2}}\!K_2 \big]\! [\dot u,\dot u ]
		\!\!+\!\! 4\rho^{-\frac{3}{2}}\!K_2 [D_y^2 u,\dot u]\!\!+\!\! \rho^{-\frac{7}{2}}\!K_2 [u,\dot u]
		\!\!+\!\! \rho^{-\frac{5}{2}}\!K_2 [(D_y^2+1)u,u]\!\!+\!\!  \frac{1}{8}\rho^{-\frac{9}{2}}\!K_2 [u,u]
		\ . \notag
\end{align}
The term $T_1$ is no better than the original quadratic
nonlinearity,  so one chooses the $K_i$ specifically to achieve 
$T_1=\rho^{-\frac{1}{2}}\alpha_0 u^2$.
A standard calculation (see Section 7.8 of \cite{H1}) shows that to do this one must define the symbols
of $K_0$ and $K_2$ as follows:
\begin{equation}
		k_0 \ = \ \alpha_0(1-2\xi\eta )q^{-1}\ , \qquad
		k_2 \ = \ 2\alpha_0 q^{-1}\ , 
		\qquad\qquad \hbox{where \ \ } q \ = \ (4\xi^2 + 4\eta^2 + 4\xi\eta + 3) \ . \notag
\end{equation}
Notice that these symbols  are such that $k_0\in S^{1,1;2}$ and $k_2\in S^{0,0;2}$.
To estimate $\mathcal{N}_{quad}$ and the error terms $T_2$ and $T_3$ listed above 
we will use:
\begin{prop}[Quadratic NF Estimates]\label{quad_error_prop}
Let 
$(\Box_\mathcal{H}+1)u=\rho^{-\frac{1}{2}}\alpha_0 u^2 + \rho^{-1}\beta u^3 + F:=G$. 
Then one has the estimate:
\begin{equation}
		\rho^\frac{1}{2} \lp{ G}{S(\rho)}  \!+\!  
		\rho^\frac{1}{2}\lp{ \dot G}{\dot S(\rho)}\!+\!
		\lp{\dot u}{\dot S(\rho)}  
		\!+\!
		\lp{\ddot u \!+\! D_y^2u }{S(\rho)}
		 \!\lesssim\!  \mathcal{Q}_1(
		 \lp{u}{S(\rho)},   \lp{F}{S(\rho)})
		\!+\!  \lp{\rho^\frac{1}{2} {F}}{ S(\rho)}
		\!+\! \lp{\rho^\frac{1}{2} \dot{F}}{\dot S(\rho)} 
		\ . \label{quad_nf_part1} 
\end{equation}
In addition if we set $G=\rho^{-\frac{1}{2}}\alpha_0 u^2 + \rho^{-1}\beta u^3 + F$
in  $T_2$ and $T_3$  above, then one has the  fixed time estimates:
\begin{align}
		\rho  \lp{ T_2}{S(\rho)} + \rho^{\frac{3}{2}} \lp{ T_3}{S(\rho)}
		\ &\lesssim \ 
		\hbox{(R.H.S.)}\eqref{main_nf_ests_quad'}
			\ , \label{quad_nf_error1}\\
		\lp{T_2 +\rho^{-1}( 2\alpha_0^2 u^3 -\frac{8}{3}\alpha_0^2\dot u^2 u)}{L^\infty_y} \ &\lesssim \ 
		\rho^{-\frac{3}{2}+\delta}\hbox{(R.H.S.)}\eqref{main_nf_ests_quad'}
		\ . \label{quad_nf_error2}
\end{align}
\end{prop}

First we shall use this Proposition to demonstrate Theorem \ref{quad_nf_thm}.

\begin{proof}[Proof that Proposition \ref{quad_error_prop} implies Theorem \ref{quad_nf_thm}]
With the choice \eqref{N_quad_defn} the estimate for $\rho^\frac{1}{2}\mathcal{N}_{quad}$
on line \eqref{main_nf_ests_quad'} is an immediate consequence of 
\eqref{bilin_est1} and \eqref{bilin_est3} and the estimate for $\dot u$ on line \eqref{quad_nf_part1}.

The identity \eqref{main_nf_iden_quad'} is satisfied if we define:
\begin{equation}
		\mathcal{R}_{quad} \ = \  T_2+T_3 - \frac{1}{4}\rho^{-2}
		\mathcal{N}_{quad} \ , \qquad
		\td{\mathcal{R}}_{quad} \ = \ (2\rho^{-1}\alpha_0^2 u^3 - \frac{8}{3}\rho^{-1}
		\alpha_0^2\dot{u}^2 u +T_2) +T_3 - \frac{1}{4}\rho^{-2}
		\mathcal{N}_{quad} \ . \notag
\end{equation}
Then the estimate for $(\mathcal{R}_{quad},\dot{\mathcal{R}}_{quad})$ on 
line \eqref{main_nf_ests_quad'} follows immediately from the second and third embeddings
on line \eqref{easy_S_to_N}, estimate \eqref{quad_nf_error1}, and the estimate for $\mathcal{N}_{quad}$
already shown. The estimate for $\td{\mathcal{R}}_{quad}$ on 
line \eqref{main_nf_ests_quad'} is immediate from the estimate for $T_3$ on
line \eqref{quad_nf_error1},
estimate \eqref{quad_nf_error2}, and  the estimate for $\mathcal{N}_{quad}$.
\end{proof}

Now we prove Proposition \ref{quad_error_prop} which is done separately for the various estimates
involved. First the preliminary bounds.

\begin{proof}[Proof of estimate \eqref{quad_nf_part1}]
We consider each term on the LHS of \eqref{quad_nf_part1} separately.

\case{1}{The $S(\rho)$ estimate for $G$}
For the part of $G$ which contains $F$ this is immediate. 
For the quadratic terms in $G$ the bound is also immediate from
the first algebra estimate on line \eqref{alg2}. For the cubic part of $G$  we can
again use \eqref{alg2} once we know $\rho^{-\frac{1}{2}}\beta_1(\rho y)\in S(\rho)$
uniformly.
The main thing to check  is the $H^1_y$ portion of the $S(\rho)$ norm which has the worst
behavior. For this  
we compute:
\begin{equation}
		\lp{\beta_1}{H^1_y} \ \lesssim \ \rho \lp{(\beta_1,\beta_1')}{L^2_y} 
		 \ \lesssim \ \rho^\frac{1}{2} \ ,
		 \qquad \hbox{where \ \ } \beta_1' = (\partial_x \beta_1) (\rho y) \ , 
		  \notag
\end{equation}
which suffices. Similar $H^1_y$ bounds for $\partial_\rho (\rho^{-\frac{1}{2}}\beta_1)$ 
and $D_y(\rho^{-\frac{1}{2}}\beta_1)$  are immediate.

\case{2}{The $\dot S(\rho)$ estimate for $\dot u$}
Using the
second bound on line \eqref{dot_S_bnds}, this boils down to estimating
$ \lp{G}{\rho^\delta H^1_y\cap L^2_y\cap L^\infty_y}$   which is 
already contained in the $S(\rho)$ bound for $G$ above.

\case{3}{The $ S(\rho)$ estimate for $\ddot u + D_y^2u$}
Since 
$\ddot u + D_y^2 u=G-(1+\frac{1}{4}\rho^{-2})u$ the bound follows from the $ S(\rho)$ 
estimates already proved for $G$ and $u$.

\case{4}{The $\dot S(\rho)$ estimate for $\dot G$}
Again it suffices to focus attention on the nonlinear terms. Differentiating the 
quadratic and cubic terms in $G$ with respect to $\partial_\rho$ and
using the second algebra
estimate on line \eqref{alg2}, the desired bound follows at once from the 
$\dot{S}(\rho)$ estimate for $\dot u$ and the $S(\rho)$ estimate for 
 $\rho^{-\frac{1}{2}}\beta_1$ proved previously.
\end{proof}

Finally, we prove the error bounds for $T_2$ and $T_3$ in Proposition \ref{quad_error_prop}.

\begin{proof}[Proof of estimates  \eqref{quad_nf_error1} and  \eqref{quad_nf_error2}]
There are a number of cases.

\case{1}{The $S(\rho)$ estimate for $ T_2$ on line \eqref{quad_nf_error1}} 
This follows immediately from
the $S(\rho)$ and $\dot S(\rho)$ bounds for $G$, $\dot G$, and $\dot u$ on line
\eqref{quad_nf_part1} and the bilinear estimates \eqref{bilin_est1}--\eqref{bilin_est3}.

\case{2}{The improved $L^\infty_y$ estimate for $T_2$ on line \eqref{quad_nf_error2}}
Here we use estimate \eqref{imp_L00_est} which shows that in light of 
\eqref{quad_nf_part1}  we already have 
estimate \eqref{quad_nf_error2} for the expression
$T_2 - \rho^{-\frac{1}{2}}\alpha_0 \big[ -2 Gu +\frac{4}{3} \dot G \dot u +\frac{4}{3} G^2 \big]$. 
Thus, we only need to prove the bound:
\begin{equation}
		\lp{  \rho^{-\frac{1}{2}}\alpha_0 \big[2 Gu -\frac{4}{3} \dot G \dot u -\frac{4}{3} G^2 \big]
		- \rho^{-1}\alpha_0^2 (2 u^3 - \frac{8}{3} \dot u^2 u)}{L^\infty_y}
		\! \lesssim \! \rho^{-\frac{3}{2}}\big(
		\mathcal{Q}_2\big(\lp{u}{S(\rho)}, \lp{\rho (F,\dot F)}{L^\infty_y}\big)
		\!+\! \lp{\rho (F,\dot F)}{L^\infty_y}^2 \big)
		 \ . \notag
\end{equation}
Expanding the formula for $G$ into the first three terms on the LHS, and canceling
the $O(\rho^{-1})$ terms, this estimate follows easily from inspection.

\case{3}{The $S(\rho)$ estimate for $ T_3$ on line \eqref{quad_nf_error1}}
The estimate for all but the first  
two terms follows immediately from \eqref{bilin_est1} and \eqref{bilin_est2}, and the 
estimate for $\dot u$ on line \eqref{quad_nf_part1}.
For the first two terms we need to expand the commutator with $\partial_\rho^2$. The important
observation here is that by formula \eqref{leib2} when  $\partial_\rho$ falls on the kernel
$K({}^1\!D, {}^2\!D)$ it preserves its $S^{a,b;c}$ class and adds an inverse power of $\rho$. Therefore,
modulo a factor of $\rho^{-\frac{3}{2}}$ or better the errors are one of four types:
\begin{align}
		K[S(\rho),S(\rho)] &\hbox{\ \ and \ } k \in S^{1,1;2} \ , \qquad
		&K[S(\rho),\dot S(\rho)] &\hbox{\ \ and \ } k \in S^{1,1;2} \ , \notag\\
		K[\dot S(\rho),\dot S(\rho)] &\hbox{\ \ and \ } k \in S^{0,0;2} \ , \qquad
		&K[\dot S(\rho),\partial_\rho \dot S(\rho)] &\hbox{\ \ and \ } k \in S^{0,0;2} \ . \notag
\end{align} 
All but the last case are covered by the general class of estimates \eqref{bilin_est1}-\eqref{bilin_est3}.
To handle the last case above we need to replace the generic marker
$\partial_\rho \dot S(\rho)$ with its actual value, i.e.~$\ddot u$. Using the estimate 
for $\ddot u + D_y^2 u$ on line \eqref{quad_nf_part1} we can trade the last expression
on the previous line for a sum of:
\begin{equation}
		K[\dot S(\rho),S(\rho)] \hbox{\ \ and \ } k \in S^{0,2;2} \ , \qquad
		K[\dot S(\rho),S(\rho)] \hbox{\ \ and \ } k \in S^{0,0;2} \ , \notag
\end{equation}
which are again of the  form covered by \eqref{bilin_est1}-\eqref{bilin_est3}.
\end{proof}

%-------------------------------------------------------------------------
%%%%%%%%%%%%%%%%%%%%%%%%%%%%%% 
%-------------------------------------------------------------------------

\section{The Cubic Correction}\label{Cubic_sect}

Our main result here is:

\begin{thm}[Cubic Normal Forms]\label{cubic_nf_thm}
Let $w$ be sufficiently smooth and well localized 
solution to the equation:
\begin{equation}
		(\Box_\mathcal{H} + 1)u \ = \     \rho^{-1}\beta_1 u^3 + F \ . \label{cubic_part}
\end{equation}
Then there exists nonlinear functionals  
$\mathcal{N}_{cubic}=\mathcal{N}_{cubic}(u,\dot u)$ and 
$\mathcal{R}_{cubic}=\mathcal{R}_{cubic}(u,\dot u,F,\dot F)$
such that the following algebraic equation holds:
\begin{equation}
		(\Box_\mathcal{H}+1)\mathcal{N}_{cubic} \ =\  \rho^{-1}\beta_1 u^3 
		+ \mathcal{R}_{cubic} \ ,
		\notag
\end{equation}
and one has the  following  estimates:
\begin{equation}
		\lp{\mathcal{N}_{cubic} }{S[1,T]} 
		+ \lp{\mathcal{R}_{cubic} }{N[1,T]}\ \lesssim \ 
		\mathcal{Q}_2\big(\lp{u}{S[1,T]}, \lp{F}{N[1,T]},\lp{\rho \dot F}{L^\infty_\rho(L^2_y)[1,T]}\big)
		+  \mathcal{Q}_3\big( \lp{F}{N[1,T]}\big) 
		\ . \label{main_cubic_nf_ests}
\end{equation}
\end{thm}

We will construct the correction as a dyadic sum 
$\mathcal{N}_{cubic}=\sum_{k> 0}   \mathcal{N}_k$ with dyadic errors
$\mathcal{R}_k$. Our ansatz for each dyadic piece is:
\begin{equation}
		 \mathcal{N}_k \ = \   \frac{1}{\rho} \big( f_{1,k} (u_{<k})^3 +
        		f_{2,k} (u_{<k})^2\dot{u}_{<k}+   f_{3,k} u_{<k}(\dot{u}_{<k})^2
         	+   f_{4,k} (\dot{u}_{<k})^3\big) \ . \label{cubic_NF_ansatz}
\end{equation}
The functions $f_{i,k}$ will be of the following type.

\begin{defn}[$S^\frac{1}{2}_N$ Symbols]
We say a symbol (function) $f(\rho,y)$ is in $S^\frac{1}{2}_N$  
if it obeys the uniform fixed time bounds:
\begin{equation}
		|\nabla^\alpha f| \ \lesssim \ \rho^{(\frac{1-|\alpha|}{2})_+} \ , \qquad
		\hbox{where\ \ } \nabla f = (\dot{f},D_y f) \ ,
		\qquad \hbox{for all\ \ }  |\alpha|\leqslant N \ .
		\label{sym}
\end{equation}
We write the associated norm   as $\llp{f(\rho)}{S^\frac{1}{2}_N}$.
\end{defn}

The first estimate we will need for such coefficients is the following,
which follows directly from \eqref{SN_mult1} and the inclusion
$S(\rho)\subseteq \dot S(\rho)$, and $\lp{\dot u}{\dot S(\rho)}\lesssim 
\mathcal{Q}_1(\lp{u}{S(\rho)})+\lp{F}{N(\rho)}$ which follows easily from  \eqref{dot_S_bnds}
and \eqref{alg2}.

\begin{lem}[$S(\rho)$ space bound for the $\mathcal{N}_k$]\label{symbol_lem}
Let $f_k \in S^\frac{1}{2}_2$ with $k>0$, and let
$u_{<k},v_{<k},w_{<k}$  be sufficiently smooth and well localized. Then one has the
product estimates:
\begin{equation}
		\lp{\rho^{-1}f_k u_{<k} v_{<k} w_{<k}}{S(\rho)} \ \lesssim \ \rho^{-\delta}
		\llp{f_k(\rho)}{S^\frac{1}{2}_2} \lp{u_{<k}}{A}\, \lp{v_{<k} }{B}\, \lp{w_{<k}}{C} \ , 
		\label{f_k_product}
\end{equation}
where $A,B,C$ denotes any combination of the norms $S(\rho)$ and $\dot S(\rho)$.
\end{lem}

In light of  \eqref{f_k_product}, to prove Theorem \ref{cubic_nf_thm}
it suffices to show:

\begin{prop}[Dyadic Cubic Normal Forms]\label{cubic_nf_thm_red}
Let $u$ be sufficiently smooth and well localized and solve equation \eqref{cubic_part}.
Then there exists functions $f_{i,k}$ such that $f_{i,k}\equiv 0$ for $2^k\ll \rho^{\frac{1}{2}}$,
and if one defines $\mathcal{N}_k$
as on line \eqref{cubic_NF_ansatz} one has:
\begin{equation}
		(\Box_\mathcal{H}+1)\mathcal{N}_k- \rho^{-1}(\beta_1)_k (u_{<k})^3
		 \ = \  \mathcal{R}_k \ ,
		\label{cubic_remainder}
\end{equation}
and one has the fixed time dyadic sum estimates:
\begin{align}
		\sum_k  \llp{f_{i,k}(\rho)}{S^\frac{1}{2}_2} \ &\lesssim \ \ln(\rho) \ , 
		  \label{fixed_time_dyad_bnds1}\\
		 \sum_k \lp{\mathcal{R}_k(\rho) }{N(\rho)} \ &\lesssim \  
		 \mathcal{Q}_2\big(\lp{u(\rho)}{S(\rho)},\lp{F(\rho)}{N(\rho)},\lp{\rho \dot F(\rho)}{L^2_y}\big) 
		 +   \mathcal{Q}_3\big( \lp{F}{N(\rho)}\big)
		 \ . \label{fixed_time_dyad_bnds2}
\end{align}
\end{prop}

We'll postpone the demonstration of this last result until the next two subsections. 
For now we use it to conclude:

\begin{proof}[Proof that Lemma \ref{symbol_lem} and Proposition \ref{cubic_nf_thm_red} 
implies Theorem \ref{cubic_nf_thm}]
First define $\mathcal{N}_{cubic}=\sum_k \mathcal{N}_k$ and then set
$\mathcal{R}_{cubic}=\rho^{-1}\beta_1u^3 - \sum_k \rho^{-1}(\beta_1)_k (u_{<k})^3
+\sum_k\mathcal{R}_k$. Thanks to \eqref{f_k_product} and \eqref{fixed_time_dyad_bnds1}
we have the bound for $\mathcal{N}_{cubic}$ on line \eqref{main_cubic_nf_ests}.
On the other hand, to prove  \eqref{main_cubic_nf_ests} for $\mathcal{R}_{cubic}$
we see that estimate \eqref{fixed_time_dyad_bnds2} reduces our task to showing at fixed time:
\begin{equation}
		\lp{\rho^{-1}\big( \beta_1 u^3-
		\sum_{k\geqslant 0 } (\beta_1)_k (u_{<k})^3 \big)}{N(\rho)} \ 
		\lesssim \  \lp{u(\rho)}{S(\rho)}^3 \ . \notag
\end{equation}
For the $L^2_y\cap L^\infty_y$ part of the $N(\rho)$ norm 
it suffices to treat the two terms separately, and one can sum using the $L^\infty_y$
version of the coefficient bound \eqref{P_k_beta}.

For the $H^1_y$ portion of the  norm, expanding the dyadic sum it suffices to show
the fixed frequency bound:
\begin{equation}
		\lp{(\beta_1)_k\big( u^3 -(u_{<k})^3\big) }{H^1_y} \ \lesssim \ 2^{-|k-\ln_2(\rho)|}
		\lp{u}{H^1_y}\lp{u}{L^\infty_y}^2 \ . \notag
\end{equation}
This follows at once by writing $u^3 -(u_{<k})^3=(u_{\geqslant k})^3
+3(u_{\geqslant k})^2 u_{<k}+3 u_{\geqslant k}( u_{<k})^2$, and then using
the Leibniz rule and again  \eqref{P_k_beta}. If the derivative falls on $P_k \beta_{1}$
we can move it to a high frequency factor via the simple estimate
$2^k\lp{u_{\geqslant k}}{L^2_y}\lesssim \lp{u}{H^1_y}$.
\end{proof}

The remainder of this section is devoted to the proof of Proposition \ref{cubic_nf_thm_red}.

%-------------------------------------------------------------------------

\subsection{Construction of the Symbols}

We now do an algebraic calculation of the equations satisfied by the coefficients
$f_{i,k}$ in the expansion \eqref{cubic_NF_ansatz}. For starters we drop the index $k$
and look for  asymptotic solutions to the equation:
\begin{equation}
		(\Box_\mathcal{H}+1)\mathcal{N} \ = \ \rho^{-1} \beta_1 u^3 \ . \label{nf_to_solve}
\end{equation}
In deciding what terms are principle, we employ the following simple rules:
\begin{itemize}
        \item Any term containing a product involving   $D_y u$ is an error.
        \item Any term with a factor of $\rho^{-2}$ is an error.
        \item Any derivative of an error remains an error.
\end{itemize}
Then lengthy but straight forward computation shows:
\begin{multline}
        (\Box_\mathcal{H} + 1)\mathcal{N} \ = \
        \frac{1}{\rho} \Big[ \, \big( \Box_\mathcal{H} f_1 - 2f_1 -2\partial_\rho
        f_2 + 2f_3 \big)\cdot u^3 
        +
        \big( \Box_\mathcal{H} f_2 - 6f_2 + 6\partial_\rho f_1
        -4\partial_\rho f_3 + 6 f_4\big)\cdot \dot{u}\, u^2\\
        +  \big( \Box_\mathcal{H} f_3 - 6f_3 + 6f_1 +
        4\partial_\rho f_2 - 6\partial_\rho f_4 \big)\cdot \dot{u}^2 u
        +  \big(\Box_\mathcal{H} f_4 - 2f_4 + 2 f_2 +
        2\partial_\rho f_3 \big)\cdot \dot{u}^3\Big]
        + \mathcal{R} \ ,
        \label{main_cubic_comp}
\end{multline}
where the remainder term can be written as
$\mathcal{R}= -2\rho^{-1} \partial_\rho \mathcal{N} +
\mathcal{R}_1+\mathcal{R}_2 + \mathcal{R}_3 + \mathcal{R}_4$
with:
\begin{align}
        \mathcal{R}_1 \! &= \! \frac{1}{\rho}
        \Big[ 3f_1 u^2  \!
        + \! 5f_3 \dot{u}^2  \!
        - \! 4 f_3 u^2  
        \! + \! 6f_2 u\dot{u}   
%        \!+\! 3f_4 \dot{u}^2 
        \!-\! 12 f_4 u\dot{u}  %\!+\! 6 f_4 \dot{u} 
        \!+\! 4\partial_\rho f_3 u\dot{u}  \!+\! 2\partial_\rho
        f_2 u^2  + 6 \partial_\rho f_4 \dot{u}^2  
        \Big] (\ddot{u}+u) \ , \label{cubic_error_1}\\
        \mathcal{R}_2 \! &= \! 
        \frac{1}{\rho}  (2f_3 u + 6f_4 \dot u)(\ddot{u}+u)^2
        + \frac{1}{\rho}
        \Big[   2f_3 u\dot{u}
         + f_2 u^2  
        + 3f_4 \dot{u}^2  
        \Big] \partial_\rho (\ddot{u}+u) 
        \ , \label{cubic_error_2}\\
         \mathcal{R}_3 \! &= \!  
         2 \frac{1}{\rho} \Big[ D_y f_1 D_y( u^3) +
         D_y f_2 D_y (u^2\dot{u})
        +  D_y f_3 D_y(u\dot{u}^2)
         +  D_y f_4 D_y (\dot{u}^3)\Big] \ , \label{cubic_error_3} \\
         \mathcal{R}_4 \! &= \!
          \frac{1}{\rho}\Big[  f_1 D_y^2( u^3) +
         f_2 D_y^2 (u^2\dot{u})
        +   f_3 D_y^2(u\dot{u}^2)
         +  f_4 D_y^2 (\dot{u}^3)\Big]
        \ . \label{cubic_error_4}
\end{align}
Setting the principle terms on the RHS of \eqref{main_cubic_comp} so as to 
match the RHS of \eqref{nf_to_solve} we have the system of equations:
\begin{align}
        &E_1:
        &\Box_\mathcal{H} f_1 - 2f_1 -2\partial_\rho f_2 + 2f_3 \ &= \
        \beta_1 \ , \notag\\
        &E_2:
        &\Box_\mathcal{H} f_2 - 6f_2 + 6\partial_\rho f_1
        -4\partial_\rho f_3 + 6 f_4 \ &= \ 0 \ , \notag\\
        &E_3:
        &\Box_\mathcal{H} f_3 - 6f_3 + 6f_1 +
        4\partial_\rho f_2 - 6\partial_\rho f_4 \ &= \ 0 \ , \notag\\
        &E_4:
        &\Box_\mathcal{H} f_4 - 2f_4 + 2 f_2 + 2\partial_\rho f_3 \ &= \
        0 \ . \notag
\end{align}
To uncover the underlying structure we form the new quantities:
\begin{align}
        F_1 \ &= \ 3f_1 + f_3 \ ,
    &F_2 \ &= \ f_1 - f_3 \ , \notag\\
    G_1 \ &= \ f_2 + 3f_4 \ ,
    &G_2 \ &= \ f_2 - f_4 \ . \notag
\end{align}
and then we take the associated linear combinations of the equations
$E_i$ above which yields the system:
\begin{align}
        &3E_1 + E_3:
        &\Box_\mathcal{H} F_1 -2\partial_\rho G_1 \ &= \ 3\beta_1  \ , \notag\\
        &E_2 + 3E_4:
        &\Box_\mathcal{H} G_1 + 2\partial_\rho F_1 \ &= \ 0 \ ,
        \notag\\
        &E_1 - E_3:
        &\Box_\mathcal{H} F_2 - 8F_2 - 6\partial_\rho G_2 \ &= \ \beta_1 \ , \notag\\
        &E_2 - E_4:
        &\Box_\mathcal{H} G_2 - 8G_2 + 6\partial_\rho F_2 \ &= \ 0 \ . \notag
\end{align}
We now complexify this system by introducing the quantities:
\begin{align}
        {K}^\dagger_1 \ &= \ F_1 + \sqrt{-1} G_1 \ ,
    &{K}^\dagger_2 \ &= \ F_2 + \sqrt{-1} G_2 \ . \notag
\end{align}
This allows us to write the last system of equations succinctly as:
\begin{align}
        \Box_\mathcal{H} {K}^\dagger_1 + 2i\partial_\rho {K}^\dagger_1 \ &= \ 3\beta_1
        \ , \notag\\
        \Box_\mathcal{H} {K}^\dagger_2 - 8{K}^\dagger_2 + 6i\partial_\rho {K}^\dagger_2 \ &= \
        \beta_1 \ . \notag
\end{align}
Finally, we introduce the gauge-transformed quantities:
\begin{align}
        K_1 \ &= \ e^{i\rho}{K}_1^\dagger \ ,
    &K_2 \ &= \ e^{3i\rho}{K}_2^\dagger \ , \notag
\end{align}
which allow us to rewrite the previous system of two complex
equations as follows:
\begin{align}
        (\Box_\mathcal{H} + 1) K_1 \ &= \   3e^{i\rho}\beta_1 \ ,
        &(\hbox{$0$ resonance equation})
        \ , \label{0_res_eq}\\
        (\Box_\mathcal{H} + 1) K_2 \ &= \   e^{3i\rho}\beta_1 \
        ,
        &(\hbox{$\pm\sqrt{8}$ resonance equation})
        \ . \label{sqrt8_res_eq}
\end{align}
%{\color{red}NEED TO GO BACK AND CHECK ALL THE ALGEBRA HERE.}

These equations only need to be solved asymptotically, so our
main result for symbols becomes:

\begin{prop}
For each $k> 0$ there exists symbols $K_{i,k}$ and errors $\mathcal{E}_{i,k}$
such that:
\begin{equation}
		 (\Box_\mathcal{H} + 1) K_{1,k} -  3e^{i\rho}\beta_{1,k} \ = \ 
		 \mathcal{E}_{1,k} \ , \qquad
		 (\Box_\mathcal{H} + 1) K_{2,k} -   e^{3i\rho}\beta_{1,k} \ = \ 
		 \mathcal{E}_{2,k} \ . \label{red_sym_eqs}
\end{equation}
Furthermore each error can be broken up into a sum 
$\mathcal{E}_{i,k}=\mathcal{E}_{i,k}^{small}+\mathcal{E}_{i,k}^{smooth}$ and one has:
\begin{align}
		 \lp{  (e^{-i\rho} \mathcal{E}_{1,k}^{smooth}, e^{-3i\rho}\mathcal{E}_{2,k}^{smooth}
		 )}{H^1_y} 
		 \ &\lesssim \ 2^{-\frac{1}{2} |k-\ln_2(\rho)|} \ , \label{reduced_symbol_bnds1}\\
		 \llp{  ( e^{-3i\rho}K_{2,k}  ,
		 \rho e^{-3i\rho}\mathcal{E}_{2,k}^{small}
		 )}{S^\frac{1}{2}_N} 
		 \ &\lesssim \ 2^{-|k-\ln_2(\rho)|}  \ . \label{reduced_symbol_bnds2}
\end{align}
In addition  for the zero resonance symbol we have:
\begin{equation}
		\lp{( \partial_\rho^n D_y^m e^{-i\rho} K_{1,k},
		\rho \partial_\rho^n D_y^m e^{-i\rho} \mathcal{E}_{1,k}^{small}) }{L^\infty_y} 
		\! \lesssim \!
		\rho^{(\frac{1-n-m}{2})_+}\!
		\begin{cases}
				2^{-|k-\frac{1}{2}\ln_2(\rho)|} \!+\! 
		 		2^{-|k-\ln_2(\rho)|} \ ,
				&(n,m)\neq(0,1);\\
				2^{-(k-\ln_2(\rho))_+} \ , &(n,m)=(0,1).
		\end{cases}\label{reduced_symbol_bnds3}
\end{equation}
In particular for the original $f_{i,k}$  
one has   $\llp{f_{i,k}}{S^\frac{1}{2}_2}\lesssim 2^{-(k-\ln_2(\rho))_+}$
which implies   estimate 
\eqref{fixed_time_dyad_bnds1}.
\end{prop}

\begin{proof}
We break the proof up into a series of steps.
 
\step{1}{Time scale decomposition and definition of $K_{i,k}$ and  $\mathcal{E}_{i,k}$}
The way we solve the equations on line 
\eqref{red_sym_eqs} is to use different information at different times. For
a fixed value of $k$ we define the smooth partition of unity in time:
\begin{equation}
		1 \ = \ h_k^{smooth} + h_k^{elliptic} + h_k^{dispersive} + h_k^{high} \ , \notag
\end{equation}
where  the supports are defined as:
\begin{align}
		\hbox{supp}(h_k^{high}) \ &= \ \{\rho \ll 2^k \} \ , 
		&\hbox{supp}(h_k^{dispersive}) \ &= \ \{\rho \approx 2^k \} \ , \notag\\
		\hbox{supp}(h_k^{elliptic}) \ &= \ \{ 2^k \ll \rho \leqslant 2^{2k} \} \ ,
		&\hbox{supp}(h_k^{smooth}) \ &= \ \{ \rho > 2^{2k} \} \ . \notag
\end{align}
We may assume these functions are chosen so that $|\partial_\rho^n h_k^\bullet|
\lesssim \rho^{-n}$. Then define:
\begin{align}
		\mathcal{E}_{1,k}^{smooth} \ &=\  -3e^{i\rho}h_k^{smooth}\beta_{1,k} \ , 
		&\mathcal{E}_{2,k}^{smooth} \ &= \ -e^{3i\rho} h_k^{smooth}\beta_{1,k} 
		\ , \label{smooth_rem_line}\\
		{K}_{1,k} \ &= \ 3e^{i\rho}(1-h_k^{smooth})D_y^{-2}\beta_{1,k} \ ,  
		&{K}_{2,k}^{nondisp} \ &= \ e^{3i\rho}(h_k^{elliptic} + h_k^{high})
		(D_y^2-8)^{-1}\beta_{1,k}
		\ .  \label{nondisp_line}
\end{align}
Here we define the operators $D_y^{-2}$ and $(D_y^{2}-8)^{-1}$ in terms of Fourier space
division by their symbols which is a smooth operation because we have cut away from the 
$0$ and $\pm \rho\sqrt{8}$ frequencies of $\beta_1$. 

We will set $K_{2,k}={K}_{2,k}^{nondisp}+{K}_{2,k}^{disp}$ where the first term on the RHS is defined 
above and the second will be produced by asymptotically solving:
\begin{equation}
		 (\Box_\mathcal{H} + 1) K^{disp}_{2} \approx   e^{3i\rho}
		 {h}_k^{disp} P_{|k-\ln_2(\rho)|<C}
		 \beta_{1} \ , 
		 \qquad \hbox{and then setting \ \ } {K}_{2,k}^{disp}=
		 \td{h}_k^{disp} P_k  K_{2}^{disp} \ .
		 \label{disp_sym_eq}
\end{equation}
The equation on the LHS above will be solved 
by reverting back to the original $(t,x)$ coordinates, and the   needed estimates
will then become a standard stationary phase calculation. Explicit details will be given shortly.
For now we simply note that constant $C>0$  in the above formula is chosen sufficiently large so that
${h}_k^{disp} P_k P_{|k-\ln_2(\rho)|<C}={h}_k^{disp} P_k $, and $\td{h}_k^{disp}$
is a another cutoff on $\rho\approx 2^k$ such that $\td{h}_k^{disp} {h}_k^{disp}={h}_k^{disp}$.

Finally, the ``small'' remainders are  defined by:
\begin{equation}
		\mathcal{E}_{1,k}^{small} \ = \  (\Box_\mathcal{H} + 1) K_{1,k} -  3e^{i\rho}
		(1-h_k^{smooth})\beta_{1,k} \ , \quad
		\mathcal{E}_{2,k}^{small} \ = \  (\Box_\mathcal{H} + 1) K_{2,k} -  e^{3i\rho}
		(1-h_k^{smooth})\beta_{1,k} \ . \label{small_rem}
\end{equation}
In particular this gives the formulas on line \eqref{red_sym_eqs}.
We now turn to estimating the quantities defined above.

\step{2}{Estimate \eqref{reduced_symbol_bnds1} for $\mathcal{E}_{i,k}^{smooth}$}
We need to show estimate \eqref{reduced_symbol_bnds1} for 
the quantities on line \eqref{smooth_rem_line}. This follows at once from \eqref{P_k_beta}
with $p=2$ and $m=0,1$, and the restriction $2^k\lesssim \rho^\frac{1}{2}$.

\step{3}{Estimate \eqref{reduced_symbol_bnds3} for  $K_{1,k}$}
For any of the time dependent frequency cutoffs we may safely ignore the
commutators  $[\partial_\rho^n, h_k^\bullet]$ as they produce inverse powers of
$\rho^{-n}$ and also restrict the relation between $\rho$ and $2^k$ to a finite number of dyadic 
frequencies at a  given time.  Concentrating on $D_y^{-2}\beta_{1,k}$ we 
directly have from \eqref{P_k_beta} the following estimate for $2^k\geqslant \rho^\frac{1}{2}$:
\begin{equation}
		\lp{\partial_\rho^n D_y^m D_y^{-2} \beta_{1,k}}{L^\infty_y} \ \lesssim \ 
		\rho^{-n-m+1}  2^{(m-1)k}2^{-N(k-\ln_2(\rho))_+} \ \lesssim \ \ \rho^{-n}
		\begin{cases}
				 \rho^{\frac{1}{2}}2^{-|k-\frac{1}{2}\ln_2(\rho)|} \ , &m=0;\\
				 2^{-(k-\ln_2(\rho))_+} \ , &m=1 ;\\
				 2^{-|k-\ln_2(\rho)|}	\ , &m\geqslant 2 . 
		\end{cases}   \label{split_K1k_est2}
\end{equation}
Note that there is no dyadic gain for $2^k\leqslant \rho$ in the 
middle case above. Together these  imply 
\eqref{reduced_symbol_bnds3} for $e^{-i\rho} K_{1,k}$.

\step{4}{Estimate \eqref{reduced_symbol_bnds2} for $K_{2,k}$}
We'll do this separately for $K_{2,k}^{nondisp}$ and $K_{2,k}^{disp}$.

\step{4a}{Contribution of ``elliptic'' and ``high'' time scales}
Here we consider the quantity $(D_y^2-8)^{-1}\beta_{1,k}$, where the frequencies
are restricted to the range $2^k\gg \rho$ or $2^k\ll\rho$. With this restriction division by $(D_y^2-8)^{-1}$
is a bounded smooth multiplier and so $(D_y^2-8)^{-1}\beta_{1,k}=\td{\beta}_{1,k}$
for some new $\td{\beta}_{1}\in\mathcal{S}$. Then by \eqref{P_k_beta} we get:
\begin{equation}
		\lp{\partial_\rho^n D_y^m  (D_y^2-8)^{-1} \beta_{1,k}}{L^\infty_y} \ \lesssim \ 
		\rho^{-n-m-1}  2^{(m+1)k}2^{-N(k-\ln_2(\rho))_+} \ \lesssim \ \rho^{-n}2^{-|k-\ln_2(\rho)|} \ .
		 \label{K2k_nondisp_est}
\end{equation}

\step{4b}{Contribution of the ``dispersive'' time scale}
Set $\td{\beta}_1(x)=P_{|k|<C}\beta_1(x)$. Then define the function $K_2^{disp}$
according to the formula: 
\begin{equation}
		K_2^{disp}(\rho,y) \! = \! \rho^\frac{1}{2}
		\chi(y)  \td{K}_2^{disp}(\rho \cosh(y),\rho \sinh(y)) \ , 
		\quad \td{K}_2^{disp}(t,x) \! = \! \int_1^t
		\frac{\sin\big((t-s)\langle D_x\rangle\big)}{\langle D_x\rangle}
		( e^{3is}h_k^{disp}(s) s^{-\frac{1}{2}} \td{\beta}_1)ds \ . \notag
\end{equation}
Here $\chi(y)$ is a cutoff of sufficiently large support that $1-|x/t|\leqslant \epsilon$
on the support of $\chi'$ where $\epsilon$ is the same as on line \eqref{main_SP_est}.
With this definition we have $h_{2^k}(s):=\chi(y)
\rho^\frac{1}{2}h_k^{disp}(s) s^{-\frac{1}{2}}$
satisfies $|\partial_s^n h_{2^k}|\lesssim  2^{-nk}$ when $\rho\approx 2^k$.
Then 
from the stationary phase estimate \eqref{main_SP_est} we have the two bounds:
\begin{align}
		|\partial_\rho^nD_y^m  \td{h}_k^{disp}P_k K_2^{disp} | \ &\lesssim \  1 \ , 
		 \label{K2k_disp_remainder_est1}\\
		\lp{\partial_\rho^nD_y^m \td{h}_k^{disp}P_k \big( (\Box_\mathcal{H}+1) 
		K_2^{disp} -e^{3it}h_k^{disp}(t)\chi(y)\sqrt{\rho/t}\td{\beta}_1(x)
		\big)}{L^\infty_y}
		 \ &\lesssim \ \rho^{-N} \ . \label{K2k_disp_remainder_est2}
\end{align}
To see \eqref{K2k_disp_remainder_est2} notice that \eqref{box_conj}
gives $(\Box_\mathcal{H}+1) K_2^{disp} = [D_y^2,\chi]\rho^\frac{1}{2} \td{K}_2^{disp}+ 
e^{3it}h_k^{disp}(t)\chi(y)\sqrt{\rho/t}\td{\beta}_1(x)$, and by the support properties of 
$\chi'(y)$ we have that $[D_y^2,\chi]\rho^\frac{1}{2} \td{K}_2^{disp}$ is rapidly decaying thanks
to the second term on RHS \eqref{main_SP_est}.
Note also that \eqref{K2k_disp_remainder_est1} in particular  
 implies \eqref{reduced_symbol_bnds2}
for $K_{2,k}^{disp}=\td{h}_k^{disp}P_k K_2^{disp}$.

\step{5}{Estimate for the ``small'' remainders}
It remains to prove bounds \eqref{reduced_symbol_bnds2} 
and \eqref{reduced_symbol_bnds3} 
for the two quantities defined on line \eqref{small_rem}.
We'll do this separately for each term.

\step{5a}{Estimate \eqref{reduced_symbol_bnds3} for $\mathcal{E}_{1,k}^{small}$}
A quick calculation shows:
\begin{equation}
		e^{-i\rho }(\Box_\mathcal{H} + 1) K_{1,k} - 3 
		(1-h_k^{smooth})\beta_{1,k} \ = \ 
		(2i\partial_\rho + \partial_\rho^2 + \frac{1}{4}\rho^{-2})
		\big[ (1-h_k^{smooth})D_y^{-2}\beta_{1,k} \big] \ . \notag
\end{equation}
Then estimate \eqref{reduced_symbol_bnds3} for the RHS of this last line follows from 
\eqref{split_K1k_est2}.

\step{5b}{Estimate \eqref{reduced_symbol_bnds2} for $\mathcal{E}_{2,k}^{small}$}
Another quick calculation shows:
\begin{equation}
		e^{-3i\rho}(\Box_\mathcal{H}+1)K_{2,k} - (1-h_k^{smooth})\beta_{1,k}
		 \ = \ T_1 + T_2 + T_3 \ , \notag
\end{equation}
where:
\begin{align}
		T_1 \ &= \ (6i\partial_\rho + \partial_\rho^2 + \frac{1}{4}\rho^{-2})
		\big[(h_k^{elliptic} + h_k^{high})
		(D_y^2-8)^{-1}\beta_{1,k}\big] 
		\ , \notag\\ 
		T_2 \ &= \ e^{-3i\rho}P_k [\partial_\rho^2, \td{h}_k^{disp}]  K_2^{disp}
		+ e^{-3i\rho}\td{h}_k^{disp}P_k \big[ (\Box_\mathcal{H}+1) 
		K_2^{disp} -e^{3it}h_k^{disp}(t)\td{\chi}(y)\td{\beta}_1(x)
		\big]
		\ , \notag\\
		T_3 \ &=\   \td{h}_k^{disp} P_k\big[ e^{3i\rho(\cosh(y)-1)} h_k^{disp}(\rho\cosh(y))
		\td{\chi}(y)\td{\beta}_1(\rho\sinh(y)) - h_k^{disp}(\rho) P_{|k-\ln_2(\rho)|<C}\beta_1
		\big] \ , \notag
\end{align}
where we are using the shorthand $\td{\chi}(y)=\sqrt{\rho/t}\chi(y)=\sech^\frac{1}{2}(y)\chi(y)$.
Indeed, using the change of coordinates \eqref{coords} we see that the last term in $T_2$
cancels the first term in $T_3$, so by choosing cutoffs with the property $\td{h}_k^{disp}P_k 
h_k^{disp} P_{|k-\ln_2(\rho)|<C}= h_k^{disp} P_k$ we have
$T_2+T_3=e^{-3i\rho}(\Box_\mathcal{H}+1)K_{2,k}^{disp}-h_k^{disp}\beta_{1,k}$. On the other hand 
the identity $T_1=e^{-3i\rho}(\Box_\mathcal{H}+1)K_{2,k}^{nondisp}- (h_k^{elliptic} + h_k^{high})\beta_{1,k}$
follows from a simple direct calculation.

Estimate \eqref{reduced_symbol_bnds2}  for $T_1$ and $T_2$ (resp) follows directly
from \eqref{K2k_nondisp_est} and 
\eqref{K2k_disp_remainder_est1}--\eqref{K2k_disp_remainder_est2} (resp). To bound
the last term we may rewrite it as $T_3=T_{31}+T_{32}$ where:
\begin{align}
		T_{31} \ &=\   \td{h}_k^{disp} P_k  \big[ e^{3i\rho(\cosh(y)-1)} h_k^{disp}(\rho\cosh(y))
		\td{\chi}(y)\big( \td{\beta}_1(\rho\sinh(y)) - \td{\beta}_1(\rho y)\big) \big]\ , \notag\\
		T_{32} \ &=\   \td{h}_k^{disp} P_k\big[ 
		P_{|k-\ln_2(\rho)|<C}\beta_1(\rho y) \cdot \big(
		e^{3i\rho(\cosh(y)-1)} h_k^{disp}(\rho\cosh(y))
		\td{\chi}(y) - h_k^{disp}(\rho)
		\big)\big] \ . \notag
\end{align}
Notice that by definition of $\td{\beta}_1(x)$ given in \textbf{Step 4b} above
and rescaling we have $P_{|k-\ln_2(\rho)|<C}\beta_1(\rho y)=\td{\beta}_1(\rho y)$.
For the term $T_{31}$ we easily obtain estimate \eqref{reduced_symbol_bnds2} 
by using estimate \eqref{coeff_approx_est}. To handle  $T_{32}$   it suffices
to show for any $\varphi\in \mathcal{S}$:
\begin{align}
		\lp{\partial_\rho^n D_y^m \big[ (
		e^{3i\rho(\cosh(y)-1)}-1
		)h_{k}^{disp}(\rho\cosh(y))\varphi(\rho y)\td{\chi}(y) \big]}{L^\infty_y} \ &\lesssim \ \rho^{-1} \ , \notag\\
		\lp{\partial_\rho^n D_y^m \big[ (
		h_{k}^{disp}(\rho\cosh(y))\td{\chi}(y) - h_k^{disp}(\rho))
		\varphi(\rho y)\big]}{L^\infty_y} \ &\lesssim \ \rho^{-1} \ , \notag
\end{align} 
which in both cases follows  by inspection when 
$|y|\gtrsim \rho^{-\frac{1}{2}}$ and 
Taylor expansions for $|y|\ll \rho^{-\frac{1}{2}}$.
\end{proof}

%-------------------------------------------------------------------------

\subsection{Estimates for the Remainders}

Our final task is to prove estimate \eqref{fixed_time_dyad_bnds2}. To do this we write
line \eqref{main_cubic_comp} above in a frequency localized form
$(\Box_\mathcal{H}+1)\mathcal{N}_k  =  \mathcal{R}_k^{coef} +\mathcal{R}_k$ 
where:
\begin{equation}
		\mathcal{R}_k^{coef} \ = \ \frac{1}{\rho}\big[ \mathcal{F}_{1,k}u_{<k}^3
		+\mathcal{F}_{2,k} \dot u_{<k} u_{<k}^2
		+ \mathcal{F}_{3,k} (\dot u_{<k})^2 u_{<k}
		+ \mathcal{F}_{4,k} (\dot u_{<k})^3
		\big] \ , \notag
\end{equation}
with the $\mathcal{F}_{i,k}$ denoting the errors generated by applying the equations
for $f_{i}$ on RHS \eqref{main_cubic_comp} to the coefficients $f_{i,k}$ defined above,
and where $\mathcal{R}_k =   
-2\rho^{-1} \partial_\rho \mathcal{N}_k +
\mathcal{R}_{1,k}+\mathcal{R}_{2,k} + \mathcal{R}_{3,k} + \mathcal{R}_{4,k}$,
with the $\mathcal{R}_{i,k}$ denoting the expressions on lines 
\eqref{cubic_error_1}--\eqref{cubic_error_4} with $f_{i}$ replaced by $f_{i,k}$
and $u,\dot u$ replaced by $u_{<k},\dot u_{<k}$. Each term will be estimated 
separately.

\case{1}{Terms in $\mathcal{R}_k^{coef}$}
Recall that by the construction of the previous subsection the $\mathcal{F}_{i,k}$
are  complex linear combinations of the errors $e^{-i\rho}\mathcal{E}_{1,k}$ and
$e^{-i3\rho}\mathcal{E}_{2,k}$ which satisfy the estimates on lines 
\eqref{reduced_symbol_bnds1}--\eqref{reduced_symbol_bnds3}.
In the case of the ``smooth'' errors we have access to a uniform $H^1_y$
bound and \eqref{fixed_time_dyad_bnds2} then follows from the Leibniz rule 
and H\"older's inequality. In the case of the ``small errors''  we can use
$S_2^\frac{1}{2}$ coefficient estimates so that
\eqref{fixed_time_dyad_bnds2} 
follows from \eqref{f_k_product} and the embedding \eqref{easy_S_to_N}.

\case{2}{The term  $\rho^{-1} \partial_\rho \mathcal{N}_k$}
The estimate \eqref{fixed_time_dyad_bnds2} in this case follows from \eqref{easy_S_to_N}
and the fact that \eqref{fixed_time_dyad_bnds1} already implied the
$\sum_k \mathcal{N}_k \in S(\rho)$ estimate.

\case{3}{The terms $\mathcal{R}_{1,k}$} 
All of the terms here are schematically of the form:
\begin{equation}
		\mathcal{R}_{1,k}^{schem} \ = \ \rho^{-1}(f_k,\dot f_k)(u_{<k},\dot u_{<k})^2
		( -D_y^2 u_{<k} + G_{<k}) \ , \qquad
		\hbox{where \ } G_{<k}= -\frac{1}{4}\rho^{-2}u_{<k} +
		\rho^{-1}P_{<k}( \beta_1 u^3)  + F_{<k}
		\ . \notag
\end{equation}
We break this into two further subcases.

\case{3a}{The term involving $G_{<k}$}
Here we use the immediate estimate: 
\begin{equation}
		\lp{ G_{<k}  }{  L^2_y\cap L^\infty_y}+ \lp{ \dot{G}_{<k}  }
		{  L^2_y} \ \lesssim \ \rho^{-1}
		\big( \mathcal{Q}_1(\lp{u}{S(\rho)})
		+ \lp{F}{N(\rho)} + \lp{\rho \dot F}{L^2_y}\big)  \ . \label{ddotu+u_bnd}
\end{equation}
Then from \eqref{SN_mult3} we get:
\begin{equation}
		\lp{\rho^{-1}(f_k,\dot f_k)(u_{<k},\dot u_{<k})^2 G_{<k}}{ N(\rho)}
		\ \lesssim \ \rho^{-\delta}
		\llp{f_k}{S^\frac{1}{2}_2} \lp{u}{S(\rho)}^2 \lp{\rho G_{<k}}{L^2_y} \ , \notag
\end{equation}
which is sufficient to sum and produce \eqref{fixed_time_dyad_bnds2} for terms
of this form.

\case{3b}{The term involving $D_y^2 u_{<k}$}
It remains to bound the portion of $\mathcal{R}_{1,k}^{schem}$ which contains
the factor $D_y^2 u_{<k}$. In this case we have two derivatives to
put on the coefficients $f_k$ so by \eqref{SN_mult2} we have:
\begin{equation}
		\lp{ \rho^{-1}(f_k,\dot f_k)(u_{<k},\dot u_{<k})^2
		D_y^2 u_{<k} }{N(\rho)} \ \lesssim \ 
		\lp{D_y^2 ( f_k, \dot  f_k)}{L^\infty_y} \lp{u}{S(\rho)}^3 \ . \label{good_shem_est}
\end{equation}
By estimates \eqref{reduced_symbol_bnds2} 
and \eqref{reduced_symbol_bnds3} the RHS above can be summed over all $k> 0$.

\case{4}{The terms $\mathcal{R}_{2,k}$}
In this case the schematic form is:
\begin{equation}
		\mathcal{R}_{2,k}^{schem}  =  \rho^{-1} f_k( u_{<k}, \dot u_{<k})
		( -D_y^2 u_{<k} + G_{<k})^2 + \rho^{-1} f_k 
		(u,\dot u)^2\partial_\rho ( -D_y^2 u_{<k} + G_{<k})
		  \ , \notag
\end{equation}
where $G_{<k}$ is as above. Expanding the product in the first RHS term,
and using a combination of estimates \eqref{SN_mult2} and \eqref{SN_mult3} (possibly by setting
$H_{<k+C}=G_{<k}^2$ in the latter case)
we have:
\begin{equation}
		\slp{ \mathcal{R}_{2,k}^{schem}  }{N(\rho)} \! \lesssim \! 
		\rho^{-\delta}
		\llp{f_k}{S^\frac{1}{2}_3} \big(
		\slp{u}{S(\rho)}^2 \slp{\rho (G_{<k},\dot G_{<k})}{L^2_y}
		+ \slp{u}{S(\rho)} \slp{\rho G_{<k}}{L^2_y\cap L^\infty_y}^2\big)+
		\slp{(D_y^2  f_k, D_y^4 f_k)}{L^\infty} \slp{u}{S(\rho)}^3 . \notag
\end{equation}
Then using \eqref{ddotu+u_bnd} above for $G_{<k}$ we have 
 \eqref{fixed_time_dyad_bnds2} for these terms.

\case{5}{The terms $\mathcal{R}_{3,k}$ and $\mathcal{R}_{4,k}$}
In these cases the two schematic forms are:
\begin{equation}
		\mathcal{R}_{3,k}^{schem}  =  \rho^{-1}D_y (f_k,\dot f_k)D_y
		\big[ (u_{<k},\dot u_{<k})^3\big]
		\ , \qquad
		\mathcal{R}_{4,k}^{schem}  =  \rho^{-1}(f_k,\dot f_k)D_y^2
		\big[(u_{<k},\dot u_{<k})^3\big]
		 \ . \notag
\end{equation}
In both cases we can use estimate \eqref{SN_mult2} to produce an estimate of the form 
\eqref{good_shem_est} above.

%-------------------------------------------------------------------------
%%%%%%%%%%%%%%%%%%%%%%%%%%%%%% 
%-------------------------------------------------------------------------

%-------------------------------------------------------------------------
%%%%%%%%%%%%%%%%%%%%%%%%%%%%%%%%%%%%%%%%%%%%%%%%%%%%%%%%%%%%%%%%%%%%%%%%%%
%%%%%%%%%%%%%%%%%%%%%%%%%%%%%%%%%%%%%%%%%%%%%%%%%%%%%%%%%%%%%%%%%%%%%%%%%%
%-------------------------------------------------------------------------


\begin{thebibliography}{99}

\bibitem[D1]{D1}
J.-M. Delort
\newblock {\em Existence globale et comportement asymptotique pour
l'\'equation de Klein-Gordon quasi lin\'eaire \`a donn\'ees petites
en dimension 1.}
\newblock Ann. Sci. \'Ecole Norm. Sup.  no. 4,
\textbf{34} (2001), 1--61.
Erratum: \emph{Global existence and asymptotic behavior for the 
quasilinear Klein-Gordon equation with small data in dimension 1}   
(French) Ann. Sci. \'Ecole Norm. Sup. (4) \textbf{39} (2006), no. 2, 335Ð345.
 

 
\bibitem[H1]{H1} L. H\"ormander \newblock
Lectures on nonlinear hyperbolic differential equations.
\newblock Springer Verlag (1997)

 
\bibitem[K1]{K1} S. Klainerman,\newblock {\em Global exsitence of small amplitude
solutions to nonlinear Klein-Gordon equations in four space
dimensions} \newblock Comm. Pure Appl. Math. \textbf{38}
(1985)631-641

%\bibitem[Ki1]{Ki1} E. Kirr

 
\bibitem[L-S1]{L-S1} H. Lindblad and A. Soffer
\newblock {\em A remark on long range scattering for the nonlinear
Klein-Gordon equation\/.}
\newblock  J. Hyperbolic Differ. Equ. 2 (2005), no. 1, 77--89


\bibitem[L-S2]{L-S2} H. Lindblad and A. Soffer
\newblock {\em  A remark on asymptotic completeness
 for the critical nonlinear Klein-Gordon equation.\/.}
\newblock   Lett. Math. Phys. {\bf 73} (2005), no. 3, 249--258.

\bibitem[L-S3]{L-S3} H. Lindblad and A. Soffer
\newblock {\em Scattering for the Klein-Gordon equation with quadratic and 
variable coefficient cubic nonlinearities.\/.}
\newblock   preprint 2013., arXiv:1307.5882

\bibitem[M-S] {Man-S} N. Manton and P. Sutcliffe
\newblock {\em Topological Solitons}
\newblock Cambridge Press, 2004

\bibitem[Sh]{Sh} J. Shatah \newblock {\em Normal forms and
quadratic nonlinear Klein-Gordon equations} \newblock Comm. Pure
Appl. Math. \textbf{38} (1985) 685-696

 

\end{thebibliography}
\end{document}